\documentclass[12pt]{amsart}
\usepackage{amsmath,amssymb,amsthm,mathtools}
\usepackage{mathrsfs,bm}
\usepackage{thmtools,thm-restate}
\usepackage{tikz-cd}
\usepackage{tikz}
\usepackage{enumitem}
\usepackage{graphicx}
\usetikzlibrary{matrix,arrows.meta,bending}
\usepackage{color}
\usepackage[numbers]{natbib}
\usepackage{booktabs}
\usepackage{longtable}
\usepackage{pgfplots}
\pgfplotsset{compat=1.14} 
\usetikzlibrary{arrows.meta,decorations.pathreplacing,positioning,calc} 
\usetikzlibrary{arrows.meta, positioning, fit, calc}

\usepackage[a4paper, left=3cm, right=3cm, top=3cm, bottom=3cm, footskip=15mm]{geometry}

\newtheorem{theorem}{Theorem}[section]

\newtheorem{lemma}[theorem]{Lemma}
\newtheorem{proposition}[theorem]{Proposition}
\newtheorem{corollary}[theorem]{Corollary}
\newtheorem{conjecture}[theorem]{Conjecture}
\theoremstyle{definition}
\newtheorem{definition}[theorem]{Definition}

\theoremstyle{remark}

\numberwithin{equation}{section}

\makeindex

\usepackage{fancyhdr}
\usepackage{calc} 

\AtBeginDocument{%
  \setlength{\textwidth}{\dimexpr\paperwidth - 6cm\relax}%
  \setlength{\oddsidemargin}{\dimexpr 3cm - 1in\relax}%
  \setlength{\evensidemargin}{\dimexpr 3cm - 1in\relax}%
  \setlength{\marginparwidth}{2.5cm}%
}

\begin{document}

\title{On addition chains and progress on the Scholz conjecture}

\author{T. Agama}
\address{Department of Mathematics, African Institute for Mathematical science, Ghana
}
\email{theophilus@aims.edu.gh/emperordagama@yahoo.com}


\date{\today}



\begin{abstract}
In this paper, we develop some new classes of methods to study the Scholz conjecture on addition chains. It turns out that the exponents of numbers of the form $2^n-1$ largely determine the length of the shortest addition chain for the number that leads to $2^n-1$. Using the carry analysis, we obtain improved upper bounds for the length of the shortest addition chains $\ell(2^n-1)$ producing $2^n-1$. In particular, we show that if $2^n-1$ has carry of degree at most 
$$
\kappa(2^n-1)=\frac{1}{2}\left(\ell(n)-\left\lfloor\frac{\log n}{\log 2}\right\rfloor+\sum \limits_{j=1}^{\lfloor \frac{\log n}{\log 2}\rfloor}\left\{\frac{n}{2^j}\right\}\right)
$$ 
then
$$
\ell(2^n-1)\leq n+1+\sum \limits_{j=1}^{\lfloor\frac{\log n}{\log 2}\rfloor}\bigg(\left\{\frac{n}{2^j}\right\}-\xi(n,j)\bigg)+\ell(n)
$$
for all $n\in \mathbb{N}$ with $n\geq 4$, where $\ell(\cdot)$ denotes the length of the shortest addition chain that leads to $\cdot$, $\{\cdot\}$ denotes the fractional part of $\cdot$ and where $\xi(n,1):=\{\frac{n}{2}\}$ with $\xi(n,2)=\{\frac{1}{2}\lfloor \frac{n}{2}\rfloor\}$ and so on.
\end{abstract}

\maketitle

\begingroup
  \setlength{\parskip}{6pt} 
  \tableofcontents
\endgroup

\section{Introduction}

Addition chains provide a classical and remarkably flexible framework for measuring the cost of repeated multiplication. Given a positive integer $n$, an addition chain for $n$ is a finite increasing sequence that begins at $1$ and reaches $n$, with each new term obtained as the sum of two earlier terms. The central quantity of interest is the minimal length $\ell(n)$ of such a chain. Since the work of Brauer, addition chains have occupied a natural place in combinatorial number theory and in algorithmic problems related to fast exponentiation and short multiplication schedules \cite{brauer1939addition,clift2011calculating}. In that setting, one seeks not only explicit constructions of short chains, but also structural principles that explain why certain integers admit especially efficient chains and how these chains can be compared, decomposed, and extended \cite{utz1953note,bahig2002some,thurber1973scholz}.\\

This paper is motivated by the Scholz conjecture, which concerns the special family of integers of the form $2^n-1$. In its classical form, the conjecture predicts a close relationship between the minimal chain length of $2^n-1$ and that of $n$, and it has stimulated a long line of work on short chains, star chains, binary structure and carry-based arguments \cite{brauer1939addition,utz1953note,thurber1973scholz,bahig2002some}. The present manuscript focuses on this same family, but it does so through a collection of new organizing ideas: the language of sub-addition chains, the notion of equivalent chains, fixed-base comparisons, special-form constructions, and carry analysis. These tools are used to build explicit chains for numbers of the form $2^n-1$ and to derive upper bounds that are designed to move the Scholz problem forward from several complementary directions.\\

A distinctive feature of the paper is the introduction of a comparison language for chains. The concepts of \emph{sub-addition chains}, \emph{equivalent addition chains}, and \emph{equivalence in a fixed base} are used to isolate how a chain can be viewed as a structured subsequence of another chain and how one chain may be transformed into another by controlled changes in the associated generators. This framework is then applied to chains for special integers, where the paper takes advantage of certain binary forms of the exponent $n$ of $2^n-1$ and related factorizations to construct chains with improved length estimates. In later sections, the same philosophy is refined in two ways: first, by introducing the ``symmetry'' and ``filling the pothole'' methods for extending incomplete constructions to full chains; and second, by studying carry phenomena, which measure how far a recursive factorization of $2^n-1$ departs from the ideal case and which ultimately control the quality of the resulting bounds.\\

\subsection{Organization of the paper} The paper is organized as follows. Section~2 introduces sub-addition chains and the notions of chain equivalence that are used throughout the paper. Section~3 studies addition chains for numbers of special form, especially $2^n-1$, and presents the symmetry method that leads to initial upper bounds for the length of a chain. Section~4 turns to arbitrary length addition chains and develops the filling-the-pothole construction, which yields a more explicit bound for chains producing $2^n-1$. Section~5 extends the discussion to addition chains of fixed degree. Section~6 combines the factor method with the pothole method to obtain progress on the Scholz conjecture in the form of refined inequalities for $\ell(2^n-1)$. Finally, Section~7 develops the carry and carry analysis needed to further sharpen the discussion and explain why the binary structure of the exponent plays such a decisive role in the problem.

\section{Sub-addition chains}

In this section, we introduce the notion of sub-addition chains.

\begin{definition}
Let $n\geq 3$. By the addition chain that leads to $n$, we mean the sequence of positive integers
\begin{align}
s_{0}=1,s_1=2,\ldots,s_{k-1},s_{k}=n\nonumber
\end{align}
where each term $s_j$~($j\geq 1$) in the sequence is the sum of two previous terms (repetition allowed), with the corresponding partition sequence
\begin{align}
2=1+1,\ldots,s_{k-1}=a_{k-1}+r_{k-1},s_{k}=a_{k}+r_{k}=n\nonumber
\end{align}
with $a_{i+1}=a_{i}+r_{i}$ and $a_{i+1}=s_i$ for $2\leq i\leq k$. We call the partition $a_{i}+r_{i}$ the $i^{th}$ \emph{generator} of the chain for $2\leq i\leq k$. We call $a_{i}$ the \emph{determiners} and $r_{i}$ the \emph{regulator} of the $i^{th}$ generator of the chain. We call the sequence $(r_{i})$ the regulators (gap sequence) of the addition chain and $(a_i)$ the determiners of the chain for $2\leq i\leq k$. The number of terms in the sequence (excluding $s_0$) is the length of the addition chain. The addition chain in the family of all addition chains that leads to $n$ with the smallest number of terms is the \emph{shortest} or \emph{optimal} or \emph{minimal} length chain leading to $n$. We denote the minimal length of a chain that leads to a target $\cdot$ by $\ell(\cdot)$. If the chain is not necessarily of minimal length, then we denote the length by $\delta(\cdot)$ and call it the \emph{arbitrary} length addition chain.
\end{definition}
\bigskip

\begin{definition}
Let the sequence $s_0=1,s_1=2,\ldots,s_{k-1},s_{k}=n$ be an addition chain that leads to $n$ with the corresponding partition sequence
\begin{align}
s_1=2=1+1,\ldots,s_{k-1}=a_{k-1}+r_{k-1},s_{k}=a_{k}+r_{k}=n.\nonumber
\end{align}
We call the sub-sequence $(s_{j_{m}})$ for $1\leq j\leq k$ and $1\leq m\leq t\leq k$ a \emph{subaddition} chain of the addition chain that leads to $n$. We say that it is a \emph{complete} sub-addition chain of the addition chain that leads to $n$ if it contains exactly the first $t$ terms of the addition chain. Otherwise, we say that it is an \emph{incomplete} subaddition chain.
\end{definition}

\subsection{Equivalent addition chains}

In this section, we introduce and study the notion of \emph{equivalence} of an addition chain producing a given number. We launch the following languages.

\begin{definition}\label{equivalence}
Let $s_0=1,s_1=2,\ldots,s_{k-1}, s_{k}=n$ be an addition chain that leads to $n$ with the corresponding partition sequence
\begin{align}
s_1=2=1+1,\ldots,s_{k-1}=a_{k-1}+r_{k-1},s_{k}=a_{k}+r_{k}=n\nonumber
\end{align}
where $a_{i}$ is the \emph{determiner} and $r_{i}$ the \emph{regulator} of the $i^{th}$ generator of the chain. Also, we let $u_0=1,u_1=2,\ldots, u_{l-1},u_{l}=m$ be an addition chain that leads to $m$ with the corresponding partition sequence
\begin{align}
u_1=2=1+1,\ldots,u_{l-1}=g_{l-1}+h_{l-1},u_{l}=g_{l}+h_{l}=m\nonumber
\end{align}
where $g_{i}$ is the \emph{determiner} and $h_{i}$ the \emph{regulator} of the $i^{th}$ generator of the chain. We say that the addition chain $s_0=1,s_1=2,\ldots,s_{k-1}, s_{k}=n$ is \emph{equivalent} to the addition chain $u_0=1,u_1=2,\ldots, u_{l-1},u_{l}=m$ if there exists a complete subaddition chain $(u_{i_{t}})$ of the chain $(u_i)$ such that for each determiner $g_{i}$ and the corresponding regulator $h_{i}$ in the subaddition chain there exists some $v_i,d_i\geq 0$--not necessarily distinct--with $v_i,d_i\in \mathbb{Z}$ such that $g_i=a_i-v_i$ and $h_i=r_i-d_i$. We call each $v_i$ and $d_i$ the \emph{stabilizer} of the determiner and the regulator of the $i^{th}$ generator of the chain. We denote equivalence by $(s_j)\Rightarrow (u_i)$.
\end{definition}
\bigskip

\begin{proposition}\label{length of equivalent addition chains}
Let $s_0=1,s_1=2,\ldots,s_{\delta(n)},s_{\delta(n)+1}=n$ and $1,2,\ldots,u_{\delta(m)},u_{\delta(m)+1}=m$ be addition chains that lead to $n$ and $m$ with length $\delta(n)$ and $\delta(m)$, respectively. If $(s_j)\Rightarrow (u_i)$, then $\delta(n)\leq \delta(m)$.
\end{proposition}

\begin{proof}
Let $1,2,\ldots,s_{\delta(n)},s_{\delta(n)+1}=n$ and $1,2,\ldots,u_{\delta(m)},u_{\delta(m)+1}=m$ be addition chains that lead to $n$ and $m$ with length $\delta(n)$ and $\delta(m)$, respectively. Suppose that $(s_j)\Rightarrow (u_i)$. By definition \ref{equivalence} there must exist a complete sub-addition chain $(u_{i_{t}})$ of the chain $(u_i)$ such that for each determiner $g_{i}$ and the corresponding regulator~(gap) $h_{i}$ in the sub-addition chain there exist some stabilizers $v_i,d_i\geq 0$--not necessarily distinct--with $v_i,d_i\in \mathbb{Z}$ such that $g_i=a_i-v_i$ and $h_i=r_i-d_i$. Since the length of the complete sub-addition chain $(u_{i_{t}})$ is at most the length of the addition $(u_i)$, and the number of terms in the complete sub-addition chain corresponds to the number of terms in the chain $(s_j)$, we deduce $\delta(n)\leq \delta(m)$.
\end{proof}

\subsection{Equivalent addition chains in a fixed base}

In this section, we introduce the notion of equivalence of an addition chain in fixed base. 

\begin{definition}
Let $(u_i)$ be an addition chain that leads to $m$ and $(s_j)$ be an addition chain that leads to $n$. We say that the addition chain $(s_j)$ is equivalent to the addition chain $(u_i)$ in base $n$ if there exists a complete sub-addition chain $(s_{j_m})$ of chain $(s_j)$ such that $(s_{j_m})\Rightarrow (u_i)$. We denote the length of the chain $(u_i)$ in \emph{base} $n$ with $\delta_n(m)$ and the length of the shortest of all such chains in \emph{base} $n$ with $\ell_n(m)$.  
\end{definition}

\section{Addition chains of numbers of special forms}

In this section, we study the addition chains that lead to a number of forms $2^n-1$.

\begin{proposition}
Let $\ell(n)$ denote the length of the \emph{shortest} addition chain leading to $n$. We have $\ell(2^n)=n$ and 
\begin{align}
\ell(2^n+1)=\ell(2^n)+1=n+1.\nonumber
\end{align}
\end{proposition}

\begin{theorem}\label{partial scholz conjecture}
Let $\ell(n)$ denote the length of the shortest addition chain leading to $n$. There exists some $G:\mathbb{N}\longrightarrow \mathbb{N}$ such that 
\begin{align}
\ell_{2^n}(2^n-1)=\ell(2^n)+1+G(n)=n+1+G(n).\nonumber
\end{align}
\end{theorem}

\begin{proof}
We construct the shortest addition chain that leads $2^n$ as 
$$
s_0=1,s_1=2,s_2=2^2,\ldots,s_{n-1}=2^{n-1},s_n=2^n
$$ 
with the corresponding partition sequence
\begin{align}
2=1+1,2+2=2^2, 2^2+2^2=2^3\ldots,2^{n-1}=2^{n-2}+2^{n-2},2^n=2^{n-1}+2^{n-1}\nonumber
\end{align}
with $a_i=2^{i-2}=r_i$ for $2\leq i\leq n+1$, where $a_i$ and $r_i$ denote the determiner and regulator (gap) of the $i^{th}$ generator of the chain. We construct a sub-addition chain of some equivalent addition chain by choosing stabilizers $v_3=0$ and $d_3=1$ for the determiner $a_3$ and the regulator $r_3$, respectively, and choose new regulators with determiners $g_i=h_i=2^{i-2}-2^{i-4}$ for all $4\leq i \leq n+1$. We obtain a complete sub-addition chain 
$$
1,2,3,6,\ldots,2^{n-2}-2^{n-4},2^{n-1}-2^{n-3},2^n-2^{n-2}
$$ 
of some equivalent addition chain producing $2^n-1$ with the corresponding partition sequence
\begin{align}
2=1+1,2+(2-1)=2^2-1, (2^2-1)+(2^2-1)=2^3-2\nonumber \\ \ldots, \nonumber \\ (2^{n-1}-2^{n-3})=(2^{n-2}-2^{n-4})+(2^{n-2}-2^{n-4}),(2^n-2^{n-2})=(2^{n-1}-2^{n-3})\nonumber \\+(2^{n-1}-2^{n-3}).\nonumber
\end{align}
By Proposition \ref{length of equivalent addition chains}, we obtain 
\begin{align}
\ell(2^n)+1<\ell_{2^n}(2^n-1)\nonumber
\end{align}
so that there exists some $G:=G(n)\in \mathbb{N}$ such that we can write $\ell_{2^n}(2^n-1)=\ell(2^n)+1+G(n)=n+1+G(n)$.
\end{proof}

\begin{corollary}\label{partial scholz 2}
Let $\ell(\cdot)$ denote the length of the shortest addition chain leading to $\cdot$. We have
\begin{align}
\ell(2^n-1)\leq n+1+G(n)\nonumber
\end{align}
where $G:\mathbb{N}\longrightarrow \mathbb{N}$.
\end{corollary}

\begin{proof}
By Theorem \ref{partial scholz conjecture}, we can write
\begin{align}
\ell(2^n-1)\leq \ell_{2^n}(2^n-1)=n+1+G(n)\nonumber
\end{align}
with $G:\mathbb{N} \longrightarrow \mathbb{N}$. 
\end{proof}
\bigskip

The corollary \ref{partial scholz 2} may be considered as a weaker and crude form of the inequality relating the length of the shortest addition producing $2^n-1$ to the length of the shortest addition chain leading to $n$. In particular, the Scholz conjecture is the claim that 
\begin{align}
G(n)\leq \ell(n)-2.\nonumber
\end{align}
Consequently, minimizing the function $G(n)$ is the ultimate goal, which also requires choosing a suitably short addition chain that completes the addition chain leading to $2^n-1$. We prove a much weaker upper bound for the function $G(n)$ with a certain construction for the addition chain that completes the addition chain in the base $2^n$.

\subsection{The symmetry method}

\begin{theorem}
We have
\begin{align}
\ell_{2^n}(2^n-1)\leq n+1+\left \lfloor \frac{n-2}{2}\right \rfloor.\nonumber
\end{align}
\end{theorem}

\begin{proof}
By Proposition \ref{partial scholz conjecture}, we can write 
\begin{align}
\ell_{2^n}(2^n-1)=n+1+G(n)\nonumber
\end{align}
where $G:\mathbb{N}\longrightarrow \mathbb{N}$ under the complete sub-addition chain 
$$
1,2,3,6,\ldots,2^{n-2}-2^{n-4},2^{n-1}-2^{n-3},2^n-2^{n-2}
$$ 
of some addition chain with corresponding partition sequence 
\begin{align}
2=1+1,2+(2-1)=2^2-1,(2^2-1)+(2^2-1)=2^3-2\nonumber \\ \ldots, \nonumber \\ (2^{n-1}-2^{n-3})=(2^{n-2}-2^{n-4})+(2^{n-2}-2^{n-4}),(2^n-2^{n-2})=(2^{n-1}-2^{n-3})\nonumber \\+(2^{n-1}-2^{n-3}).\nonumber
\end{align}
We note that we can extend the complete sub-addition chain to the addition chain leading to $2^n-1$ as follows:
\begin{align}
1,2,3,6,\ldots,2^{n-2}-2^{n-4},2^{n-1}-2^{n-3},2^n-2^{n-2},2^n-2^{n-4},2^{n}-2^{n-6},\ldots,2^n-2^{0}\nonumber \\=2^n-1\nonumber
\end{align}
by adjoining the corresponding partition sequence to the partition sequence of the complete sub-addition chain
\begin{align}
(2^n-2^{n-2})+(2^{n-2}-2^{n-4}),(2^{n}-2^{n-4})+(2^{n-4}-2^{n-6}),\ldots (2^n-2^2)+(2^2-1).\nonumber
\end{align}
We deduce
\begin{align}
G(n)\leq \left\lfloor\frac{n-2}{2}\right\rfloor \nonumber
\end{align}
where $G(n)$ counts the number of terms adjoined to the complete sub-addition chain before $2^n-1$.
\end{proof}

\section{Arbitrary length addition chains}

In this section, we study the Scholz conjecture. We prove an inequality related to the conjecture on additions chains. We begin with the following fundamental result which can be found in \cite{brauer1939addition}.

\begin{lemma}\label{braurer lower bound}
Let $\ell(n)$ denote the shortest length of an addition chain leading to $n$. We have
\begin{align}
\ell(n)>\frac{\log n}{\log 2}-1.\nonumber
\end{align}
\end{lemma}
\bigskip

\begin{theorem}\label{Apartial scholz conjecture}
Let $\delta(\cdot)$ denote the length of an addition chain that leads to $\cdot$. There exists some $K:\mathbb{N}\longrightarrow \mathbb{R}$ such that 
\begin{align}
\delta(2^n-1)\leq n-1+\left\lfloor\frac{\log n}{\log 2}\right\rfloor+K(n).\nonumber
\end{align}
\end{theorem}

\begin{proof}
We construct the shortest addition chain that leads to $2^n$
$$
1,2,2^2,\ldots,2^{n-1},2^n
$$ 
with the corresponding partition sequence 
\begin{align}
2=1+1,2+2=2^2, 2^2+2^2=2^3 \ldots,2^{n-1}=2^{n-2}+2^{n-2},2^n=2^{n-1}+2^{n-1}\nonumber
\end{align}
with $a_i=2^{i-2}=r_i$ for $2\leq i\leq n+1$, where $a_i$ and $r_i$ denote the determiner and regulator of $i^{th}$ generator of the chain. We consider only the complete sub-addition chain 
\begin{align}
2=1+1,2+2=2^2,2^2+2^2=2^3 \ldots,2^{n-1}=2^{n-2}+2^{n-2}.\nonumber
\end{align}
We extend this complete sub-addition chain by adjoining the sequence 
\begin{align}
2^{n-1}+2^{\lfloor\frac{n-1}{2}\rfloor},2^{n-1}+2^{\lfloor \frac{n-1}{2}\rfloor}+2^{\lfloor \frac{n-1}{2^2}\rfloor}\ldots,2^{n-1}+2^{\lfloor\frac{n-1}{2}\rfloor}+2^{\lfloor\frac{n-1}{2^2}\rfloor}+\cdots+2^1.\nonumber
\end{align}
The adjoined sequence contributes at most 
\begin{align}
\left\lfloor\frac{\log n}{\log 2}\right\rfloor\nonumber
\end{align}
terms to the original complete sub-addition chain. Using
\begin{align}
2^{\lfloor\frac{n-1}{2}\rfloor}+2^{\lfloor\frac{n-1}{2^2}\rfloor}+\cdots+2^1&<\sum \limits_{j=0}^{n-1}2^j\nonumber \\&=2^n-1\nonumber
\end{align}
there exists some $K:\mathbb{N}\longrightarrow \mathbb{R}$ counting the terms in the remaining terms of the addition chain leading to $2^n-1$.
\end{proof}
\bigskip

\subsection{Explicit upper bound}

In this section, we prove an explicit upper bound for the length of an addition chain--not necessarily the shortest--leading to numbers of the form $2^n-1$. In particular, we study the problem of constructing an addition chain whose length meets a specified threshold requirement.

\begin{theorem}\label{scholz conjecture}
Let $\delta(\cdot)$ denote the length of an addition chain that leads to $n$. There exists an addition chain that leads to $2^n-1$ such that
\begin{align}
\delta(2^n-1)\lesssim n+\frac{n}{\log n}+1.3\log n\int \limits_{2}^{\frac{n-1}{2}}\frac{dt}{\log^3t}+\xi(n)\nonumber
\end{align}
where $\xi:\mathbb{N}\longrightarrow \mathbb{R}$.
\end{theorem}

\begin{proof}
We construct the shortest addition chain that leads to $2^n$ 
$$
1,2,2^2,\ldots,2^{n-1},2^n
$$ 
with the corresponding partition sequence
\begin{align}
2=1+1,2+2=2^2,2^2+2^2=2^3 \ldots,2^{n-1}=2^{n-2}+2^{n-2},2^n=2^{n-1}+2^{n-1}\nonumber
\end{align}
with $a_i=2^{i-2}=r_i$ for $2\leq i\leq n+1$, where $a_i$ and $r_i$ denote the determiner and regulator (gap) of the $i^{th}$ generator of the chain. We consider only the complete sub-addition chain 
\begin{align}
2=1+1,2+2=2^2,2^2+2^2=2^3 \ldots,2^{n-1}=2^{n-2}+2^{n-2}.\nonumber
\end{align}
and extend this complete sub-addition chain by adjoining the sequence 
\begin{align}
2^{n-1}+2^{\lfloor\frac{n-1}{2}\rfloor},2^{n-1}+2^{\lfloor\frac{n-1}{2}\rfloor}+2^{\lfloor\frac{n-1}{2^2}\rfloor}\ldots,2^{n-1}+2^{\lfloor\frac{n-1}{2}\rfloor}+2^{\lfloor\frac{n-1}{2^2}\rfloor}+\cdots+2^1.\nonumber
\end{align}
The adjoined sequence contributes at most 
\begin{align}
\left\lfloor\frac{\log n}{\log 2}\right\rfloor\nonumber
\end{align}
terms to the original complete sub-addition chain. Due to the inequality
\begin{align}
2^{n-1}+2^{\lfloor \frac{n-1}{2}\rfloor}+2^{\lfloor \frac{n-1}{2^2}\rfloor}+\cdots+2^1&<\sum \limits_{i=1}^{n-1}2^i\nonumber \\&=2^n-1\nonumber
\end{align}
we make the substitution 
\begin{align}
    R_2(n):=2^{n-1}+2^{\lfloor\frac{n-1}{2}\rfloor}+2^{\lfloor \frac{n-1}{2^2}\rfloor}+\cdots+2^1\nonumber
\end{align}
and extend the addition chain by further adjoining the sequence
\begin{align}
  R_2(n)+2^{\lfloor\frac{n-1}{3}\rfloor}, R_2(n)+2^{\lfloor \frac{n-1}{3}\rfloor}+2^{\lfloor\frac{n-1}{3^2}\rfloor},\ldots, R_2(n)+2^{\lfloor\frac{n-1}{3}\rfloor}+2^{\lfloor\frac{n-1}{3^2}\rfloor}+\cdots+2^1.\nonumber
\end{align}
The adjoined sequence contributes at most 
\begin{align}
\left\lfloor\frac{\log n}{\log 3}\right\rfloor\nonumber
\end{align}
terms to the original complete sub-addition chain. Due to the inequality
\begin{align}
   R_2(n)+2^{\lfloor\frac{n-1}{3}\rfloor}+2^{\lfloor\frac{n-1}{3^2}\rfloor}+\cdots+2^1<2^{n}-1\nonumber 
\end{align}
we continue the extension of the addition chain by using all the primes $p\leq \frac{n-1}{2}$, so that by induction the number of terms adjoined to the original complete sub-addition chains is the sum 
\begin{align}
    \sum \limits_{p\leq \frac{n-1}{2}}\left\lfloor\frac{\log n}{\log p}\right\rfloor&=\left\lfloor\frac{\log n}{\log 2}\right\rfloor+\sum \limits_{3\leq p\leq \frac{n-1}{2}}\left\lfloor\frac{\log n}{\log p}\right\rfloor \nonumber \\& \leq \left\lfloor\frac{\log n}{\log 2}\right\rfloor+\log n\sum \limits_{3\leq p\leq \frac{n-1}{2}}\frac{1}{\log p}.\nonumber
\end{align}
We deduce
\begin{align}
   \sum \limits_{3\leq p\leq \frac{n-1}{2}}\frac{1}{\log p}&=\int \limits_{2}^{\frac{n-1}{2}}\frac{d\pi(u)}{\log u}\nonumber \\&=\frac{\pi(\frac{n-1}{2})}{\log(\frac{n-1}{2})}-\frac{\pi(2)}{\log 2}+\int \limits_{2}^{\frac{n-1}{2}}\frac{\pi(t)}{t\log^2t}dt \nonumber \\& \lesssim \frac{n}{\log^2n}-\frac{1}{\log 2}+1.3\int \limits_{2}^{\frac{n-1}{2}}\frac{1}{\log^3t}dt.\nonumber
\end{align}
\end{proof}
\bigskip

\subsection{Filling the Pothole method}

In this section, we describe the method of \emph{filling the potholes} which is used to obtain an upper bound for an arbitrary length of a chain.

\begin{itemize}

\item  We first construct a complete sub-addition chain producing $2^n-1$. For technical reasons which will become clear later, we stop the chain prematurely at $2^{n-1}$.
\bigskip

\item  We extend this addition chain by a length of logarithmic order. 
\bigskip

\item This extension has missing terms to qualify as an addition chain producing $2^n-1$. We fill in the missing terms, thus obtaining what one might refer to as a spoof addition chain that leads to $2^n-1$.
\bigskip

\item Creating this spoof addition chain comes at a cost. The remaining step will be to cover the cost and render an account to obtain the upper bound.
\end{itemize}

\begin{lemma}\label{Alfred Braurer}
Let $\ell(n)$ denote the shortest addition chain that leads to $n$. If $a,b\in \mathbb{N}$, then 
$$
\ell(ab)\leq \ell(a)+\ell(b).
$$ 
\end{lemma}

\begin{proof}
The proof of this Lemma can be found in \cite{brauer1939addition}.
\end{proof}

\begin{theorem}\label{Aexceptional scholz conjecture}
There exists an addition with length satisfying
\begin{align}
\delta(2^n-1)\leq \frac{3}{2}n-\left\lfloor\frac{n-2}{2^{\lfloor \frac{\log n}{\log 2}-1\rfloor+1}}\right\rfloor+1+\frac{1}{4}(1-(-1)^n)\nonumber
\end{align}
holds for all $n\in \mathbb{N}$ with $n\geq 2$, where $\lfloor\cdot\rfloor$ denotes the floor function.
\end{theorem}

\begin{proof}
We consider the number $2^n-1$ and examine the length of the addition chain according to the parity of the exponents $n$. If $n\equiv 0\mod 2$, then 
$$
2^n-1=(2^{\frac{n}{2}}-1)(2^{\frac{n}{2}}+1).
$$ 
Setting $\frac{n}{2}=k$, we construct the addition chain that leads to $2^{k}$
$$
1,2,2^2,\ldots,2^{k-1},2^k
$$ 
with the corresponding partition sequence
\begin{align}
2=1+1,2+2=2^2,2^2+2^2=2^3 \ldots,2^{k-1}=2^{k-2}+2^{k-2},2^k=2^{k-1}+2^{k-1}\nonumber
\end{align}
with $a_i=2^{i-2}=r_i$ for $2\leq i\leq k+1$, where $a_i$ and $r_i$ denote the determiner and regulator of the $i^{th}$ generator of the chain. We consider only the complete sub-addition chain 
\begin{align}
2=1+1,2+2=2^2,2^2+2^2=2^3 \ldots,2^{n-1}=2^{k-2}+2^{k-2}\nonumber
\end{align}
and extend this complete sub-addition chain by adjoining the sequence 
\begin{align}
2^{k-1}+2^{\lfloor\frac{k-1}{2}\rfloor},2^{k-1}+2^{\lfloor\frac{k-1}{2}\rfloor}+2^{\lfloor\frac{k-1}{2^2}\rfloor}\ldots,2^{k-1}+2^{\lfloor\frac{k-1}{2}\rfloor}+2^{\lfloor\frac{k-1}{2^2}\rfloor}+\cdots+2^1.\nonumber
\end{align}
The adjoined sequence contributes at most 
\begin{align}
\left\lfloor\frac{\log k}{\log 2}\right\rfloor =\left\lfloor \frac{\log n-\log 2}{\log 2}\right\rfloor < \left\lfloor\frac{\log n}{\log 2}\right\rfloor\nonumber
\end{align}
terms to the original complete sub-addition chain. Due to the inequality
\begin{align}
2^{k-1}+2^{\lfloor\frac{k-1}{2}\rfloor}+2^{\lfloor\frac{k-1}{2^2}\rfloor}+\cdots+2^1&<\sum \limits_{i=1}^{k-1}2^i\nonumber \\&=2^k-2\nonumber
\end{align}
we insert terms into the sum
\begin{align}
2^{k-1}+2^{\lfloor\frac{k-1}{2}\rfloor}+2^{\lfloor\frac{k-1}{2^2}\rfloor}+\cdots+2^1\label{pbroken road}
\end{align}
so that we have 
\begin{align}
\sum \limits_{i=1}^{k-1}2^i=2^k-2.\nonumber 
\end{align}
We analyze the cost of filling in the missing terms of the underlying sum. We have to insert 
$$
2^{k-2}+2^{k-3}+\cdots +2^{\lfloor \frac{k-1}{2}\rfloor+1}
$$ 
into \eqref{pbroken road} and the number of new terms adjoined is at most   
\begin{align}
k-2-\left\lfloor\frac{k-1}{2}\right\rfloor.\nonumber
\end{align}
The last term of the adjoined sequence is
\begin{align}
2^{k-1}+(2^{k-2}+2^{k-3}+\cdots+2^{\lfloor \frac{k-1}{2}\rfloor+1})+2^{\lfloor \frac{k-1}{2}\rfloor}+2^{\lfloor \frac{k-1}{2^2}\rfloor}+\cdots+2^1.\label{pbroken road 1}
\end{align}
Again, we have to insert $2^{\lfloor\frac{k-1}{2}\rfloor-1}+\cdots+2^{\lfloor\frac{k-1}{2^2}\rfloor+1}$ into \eqref{pbroken road 1} and the number of terms adjoined is at most  
\begin{align}
\left\lfloor\frac{k-1}{2}\right\rfloor-\left\lfloor\frac{k-1}{2^2}\right\rfloor-1.\nonumber
\end{align}
The last term of the adjoined sequence is
\begin{align}
2^{k-1}+(2^{k-2}+2^{k-3}+\cdots +2^{\lfloor \frac{k-1}{2}\rfloor+1})+2^{\lfloor \frac{k-1}{2}\rfloor}+(2^{\lfloor \frac{k-1}{2}\rfloor-1}+\cdots+2^{\lfloor \frac{k-1}{2^2}\rfloor+1})+2^{\lfloor \frac{k-1}{2^2}\rfloor}+\nonumber \\ \cdots+2^1.\label{pbroken road 2}
\end{align}
Repeating the process, we insert in the immediately previous term by inserting into \eqref{pbroken road 2} and the number of terms adjoined is at most
\begin{align}
\left\lfloor\frac{k-1}{2^{s-1}}\right\rfloor-\left\lfloor\frac{k-1}{2^s}\right\rfloor-1\nonumber
\end{align}
for $1\leq s\leq \lfloor \frac{\log n}{\log 2}-1\rfloor $, since we are filling at most $\lfloor\frac{\log k}{\log 2}\rfloor$ blocks with $k=\frac{n}{2}$. It follows that the contribution of these new terms is at most
\begin{align}
k-1-\left \lfloor \frac{k-1}{2^{\lfloor \frac{\log k}{\log 2}\rfloor}}\right\rfloor-\left\lfloor\frac{\log k}{\log 2}\right\rfloor 
\end{align}
obtained by adding the numbers in the chain 
\begin{align}
k-1-\left\lfloor\frac{k-1}{2}\right\rfloor-1 \nonumber
\end{align}
\begin{align}
\left\lfloor\frac{k-1}{2}\right\rfloor-\left\lfloor\frac{k-1}{2^2}\right\rfloor-1\nonumber
\end{align}
\begin{align}
\vdots \vdots \vdots \vdots \vdots \vdots \vdots \vdots \vdots \vdots \vdots \vdots \nonumber
\end{align}
\begin{align}
\vdots \vdots \vdots \vdots \vdots \vdots \vdots \vdots \vdots \vdots \vdots \vdots \nonumber
\end{align}
\begin{align}
\left\lfloor\frac{k-1}{2^{\lfloor \frac{\log k}{\log 2}\rfloor-1}}\right\rfloor-\left\lfloor\frac{k-1}{2^{\lfloor\frac{\log k}{\log 2}\rfloor}}-1\right\rfloor.\nonumber
\end{align}
By Lemma \ref{Alfred Braurer}, we get
$$
\ell(2^n-1)\leq \ell(2^{\frac{n}{2}}-1)+\ell(2^{\frac{n}{2}}+1)\leq \delta(2^{\frac{n}{2}}-1)+\ell(2^{\frac{n}{2}}+1)
$$ 
for even $n$, where $\delta(\cdot)$ is the length of the constructed addition chain. By a quick book-keeping, we deduce with $k=\frac{n}{2}$ for the length
\begin{align}
\delta(2^{k}-1)&\leq k+k-1-\left\lfloor \frac{k-1}{2^{\lfloor \frac{\log k}{\log 2}\rfloor}}\right\rfloor-\left\lfloor\frac{\log k}{\log 2}\right\rfloor+\left\lfloor\frac{\log n}{\log 2}\right\rfloor\nonumber \\&=n-1-\left\lfloor \frac{n-2}{2^{\lfloor\frac{\log n}{\log 2}-1\rfloor+1}}\right \rfloor+1.\nonumber
\end{align}
Now, we construct an addition chain that leads to $2^k+1$. We construct the addition chain leading to $2^{k}$ 
$$
1,2,2^2,\ldots,2^{k-1},2^k
$$
with the corresponding partition sequence
\begin{align}
2=1+1,2+2=2^2, 2^2+2^2=2^3\ldots,2^{k-1}=2^{k-2}+2^{k-2},2^k=2^{k-1}+2^{k-1}\nonumber
\end{align}
with $a_i=2^{i-2}=r_i$ for $2\leq i\leq k+1$, where $a_i$ and $r_i$ denote the determiner and regulator of the $i^{th}$ generator of the chain. Adding $1$ to the last term of the chain, we obtain the addition chain that leads to $2^k+1$ 
$$
1,2,2^2,\ldots,2^{k-1},2^k, 2^k+1
$$
of length $k+1=\frac{n}{2}+1$. Combining the contribution of the length of the addition chains constructed, we obtain in the case $n\equiv 0\pmod 2$ the inequality 
$$
\delta(2^n-1)\leq \frac{3}{2}n-\left\lfloor\frac{n-2}{2^{\lfloor \frac{\log n}{\log 2}-1\rfloor+1}}\right\rfloor-\left\lfloor \frac{\log n}{\log 2}-1\right\rfloor.
$$
We now examine the case $n\equiv 1\pmod 2$. In this case, we write 
$$
2^n-1=(2^{n-1}-1)+(2^{n-1}-1)+1
$$ 
and construct an addition chain that leads to $2^{n-1}-1$. Once this addition chain is obtained, we add the term $2^{n-1}-1$ to itself and finally add $1$ to obtain the addition chain that leads to $2^n-1$. From this construction, the length $\delta(2^n-1)$ is the sum of the length of the addition chain $\delta(2^{n-1}-1)$ and $2$. Since $n-1\equiv 0\pmod 2$, we can adapt the argument of the even case to obtain the upper bound 
$$
\delta(2^{n-1}-1)\leq \frac{3}{2}(n-1)-\left\lfloor\frac{n-3}{2^{\lfloor\frac{\log n}{\log 2}-1\rfloor+1}}\right\rfloor+1
$$ 
The claimed inequality follows by combining both even and odd cases.
\end{proof}

\section{Addition chains of fixed degree}

In this section, we introduce the notion of an addition chain of degree $d$ and their corresponding sub-addition chains.

\begin{definition}
Let $n\geq 3$. By addition chain of fixed degree $d$ and of length $k$ that leads to $n$, we mean the sequence of positive integers
\begin{align}
s_0=1,\ldots,s_{k}=n\nonumber
\end{align}
where each term $s_j$~($j\geq 3$) in the sequence is the sum of at most $d$ previous terms (repetition allowed), with the corresponding partition sequence
\begin{align}
2=1+1,\ldots,s_{k-1}=a_{k-1}+r_{k-1},s_{k}=a_{k}+r_{k}=n\nonumber
\end{align}
with $a_{i+1}=a_{i}+r_{i}$ and $a_{i+1}=s_i$ for $1\leq i\leq k$. We call the partition $a_{i}+r_{i}$ the $i^{th}$ \emph{generator} of the chain for $1\leq i\leq k$.
\end{definition}

\begin{theorem}\label{analogous scholz conjecture}
Let $\ell(n)$ and $\ell^{\lfloor\frac{n-1}{2}\rfloor}(n)$ denote the length of the shortest addition chain and the degree $\lfloor \frac{n-1}{2}\rfloor$ addition chain, respectively, leading to $n$. We have
\begin{align}
\ell^{\lfloor \frac{n-1}{2}\rfloor}(2^n-1)\leq n+\left\lfloor\frac{\log n}{\log 2}\right\rfloor.\nonumber
\end{align}
\end{theorem}

\begin{proof}
We construct the shortest addition chain that leads to $2^n$
$$
1,2,2^2,\ldots,2^{n-1},2^n
$$ 
with the corresponding partition sequence
\begin{align}
2=1+1,2+2=2^2,2^2+2^2=2^3\ldots,2^{n-1}=2^{n-2}+2^{n-2},2^n=2^{n-1}+2^{n-1}\nonumber
\end{align}
with $a_i=2^{i-2}=r_i$ for $2\leq i\leq n+1$, where $a_i$ and $r_i$ denote the determiner and regulator of the $i^{th}$ generator of the chain. We consider only the complete sub-addition chain $1,2,2^2,\ldots,2^{n-1}$ and extend this addition chain by adjoining the sequence 
\begin{align}
2^{n-1}+\cdots+2^{t_1}+\cdots+2^{\lfloor\frac{n-1}{2}\rfloor}, 2^{n-1}+\cdots+2^{t_1}+\cdots+2^{\lfloor\frac{n-1}{2}\rfloor}+\cdots+2^{t_2}+\cdots+2^{\lfloor\frac{n-1}{2^2}\rfloor},\\ \ldots,2^{n-1}+2^{n-2}+2^{n-3}+\cdots+2^1.\nonumber
\end{align}
where $\lfloor\frac{n-1}{2}\rfloor+1\leq t_1\leq n-2$, $\lfloor \frac{n-1}{2^2}\rfloor+1\leq t_2\leq \lfloor \frac{n-1}{2}\rfloor-1,\ldots, \lfloor \frac{n-1}{2^k}\rfloor+1\leq t_k\leq \lfloor \frac{n-1}{2^{k-1}}\rfloor-1$. The terms adjoined to the constructed chain can be written in the form 
\begin{align}
2^{n-1}+\cdots+2^{t_1}+\cdots+2^{\lfloor\frac{n-1}{2}\rfloor}=2^n-2^{\lfloor\frac{n-1}{2}\rfloor}\nonumber
\end{align}
and 
\begin{align}
2^{n-1}+\cdots+2^{t_1}+\cdots+2^{\lfloor\frac{n-1}{2}\rfloor}+\cdots+2^{t_2}+\cdots+2^{\lfloor\frac{n-1}{2^2}\rfloor}=2^n-2^{\lfloor\frac{n-1}{2^2}\rfloor}.\nonumber
\end{align}
By induction, we write
\begin{align}
2^{n-1}+2^{n-2}+2^{n-3}+\cdots+2^1=2^n-2\nonumber
\end{align}
and we obtain the degree $\lfloor \frac{n-1}{2}\rfloor$ addition chain
\begin{align}
1,2,2^2,\ldots,2^{n-1},2^n-2^{\lfloor \frac{n-1}{2}\rfloor},2^n-2^{\lfloor\frac{n-1}{2^2}\rfloor},\ldots,2^n-2,2^n-1.\nonumber
\end{align}
The number of terms adjoined to the sequence contributes at most 
\begin{align}
\left\lfloor\frac{\log n}{\log 2}\right\rfloor.\nonumber
\end{align}
This completes the proof.
\end{proof}

\section{Progress on the Scholz conjecture}

In this section, we make progress on the Scholz conjecture on addition chains. We combine the factor method and the \emph{fill in the pothole} method to study short addition chains that lead to numbers of the form $2^n-1$ and the Scholz conjecture. For a number $2^n-1$, we obtain the decomposition  
$$
2^n-1=(2^{\frac{n-(1-(-1)^n)\frac{1}{2}}{2}}-1)(2^{\frac{n-(1-(-1)^n\frac{1}{2})}{2}}+1)+\frac{(1-(-1)^n)}{2}(2^{n-(1-(-1)^n)\frac{1}{2}})
$$ 
which eventually yields the following decomposition $2^n-1=(2^{\frac{n}{2}}-1)(2^{\frac{n}{2}}+1)$ in the case $n\equiv 0\pmod 2$ and 
$$
2^n-1=(2^{\frac{n-1}{2}}-1)(2^{\frac{n-1}{2}}+1)+2^{n-1}
$$ 
in the case $n\equiv 1\pmod 2$. We can iterate this decomposition up to a certain desired frequency $s$ and apply the factor method on all the factors obtained from this decomposition. We can then apply the pothole method to obtain a bound for the shortest addition chain that leads to the only factor of form $2^v-1$. The length of the shortest addition chains of numbers of the form $2^v+1$ is easy to construct: construct the shortest addition chain that leads to $2^v$, add the first term of the chain to the last term, and adjoin to the chain.
\bigskip

\begin{theorem}\label{exceptional scholz conjecture}
We have
\begin{align}
\ell(2^n-1)\leq n+1-\sum \limits_{j=1}^{\lfloor\frac{\log n}{\log 2}\rfloor}\xi(n,j)+3\left\lfloor\frac{\log n}{\log 2}\right\rfloor \nonumber
\end{align}
for all $n\in \mathbb{N}$ with $n\geq 4$, where $\ell(\cdot)$ denotes the length of the shortest addition chain that leads to $\cdot$ and where $\xi(n,1):=\{\frac{n}{2}\}$ with $\xi(n,2)=\{\frac{1}{2}\lfloor \frac{n}{2}\rfloor\}$ and so on, with $\{\cdot\}$ denoting the fractional part of a real number.
\end{theorem}

\begin{proof}
  We write 
  $$
  2^n-1=(2^{\frac{n-(1-(-1)^n)\frac{1}{2}}{2}}-1)(2^{\frac{n-(1-(-1)^n\frac{1}{2})}{2}}+1)+\frac{(1-(-1)^n)}{2}(2^{n-(1-(-1)^n)\frac{1}{2}}).
  $$ 
  We can recover the general factorization of $2^n-1$ from this identity according to the parity of the exponent $n$. In particular, if $n\equiv 0\pmod 2$, then
  $$
  2^n-1=(2^{\frac{n}{2}}-1)(2^{\frac{n}{2}}+1)
  $$ 
  and 
  $$
  2^n-1=(2^{\frac{n-1}{2}}-1)(2^{\frac{n-1}{2}}+1)+2^{n-1}
  $$ 
  if $n\equiv 1\pmod 2$. Combining both cases, we obtain
  $$
  \ell(2^n-1)\leq \ell((2^{\frac{n-(1-(-1)^n)\frac{1}{2}}{2}}-1)(2^{\frac{n-(1-(-1)^n\frac{1}{2})}{2}}+1))+2
  $$ 
  obtained by constructing an addition chain that leads to $2^{n-1}-1$, adding $2^{n-1}-1$ to $2^{n-1}-1$, adding $1$ and adjoining the result in the case $n\equiv 1\pmod 2$. Applying the lemma \ref{Alfred Braurer}, we further deduce
\begin{align}
\ell(2^n-1)\leq \ell(2^{\frac{n-(1-(-1)^n)\frac{1}{2}}{2}}-1)+\ell(2^{\frac{n-(1-(-1)^n)\frac{1}{2}}{2}}+1)+2\label{progressfirst decomp}
\end{align}
We set $\frac{n-(1-(-1)^n)\frac{1}{2}}{2}=k$ in \eqref{progressfirst decomp} and write 
$$
2^k-1=(2^{\frac{k-(1-(-1)^k)\frac{1}{2}}{2}}-1)(2^{\frac{k-(1-(-1)^k\frac{1}{2})}{2}}+1)+\frac{(1-(-1)^k)}{2}(2^{k-(1-(-1)^k)\frac{1}{2}}).
$$ 
We can recover the general factorization of $2^k-1$ from this identity according to the parity of the exponent $k$. In particular, if $k\equiv 0\pmod 2$, then we have 
$$
2^k-1=(2^{\frac{k}{2}}-1)(2^{\frac{k}{2}}+1)
$$
and 
$$
2^k-1=(2^{\frac{k-1}{2}}-1)(2^{\frac{k-1}{2}}+1)+2^{k-1}
$$ 
if $k\equiv 1\pmod 2$. Combining both cases, we get
$$
\ell(2^k-1)\leq \ell((2^{\frac{k-(1-(-1)^k)\frac{1}{2}}{2}}-1)(2^{\frac{k-(1-(-1)^k\frac{1}{2})}{2}}+1))+2
$$ 
obtained by constructing an addition chain that leads to $2^{k-1}-1$, adding $2^{k-1}-1$ to $2^{k-1}-1$, adding $1$ and adjoining the result in the case $k\equiv 1\pmod 2$. Applying the lemma \ref{Alfred Braurer}, we further deduce
\begin{align}
\ell(2^k-1)&\leq \ell(2^{\frac{k-(1-(-1)^k)\frac{1}{2}}{2}}-1)+\ell(2^{\frac{k-(1-(-1)^k)\frac{1}{2}}{2}}+1)+2\nonumber \\&=\ell(2^{\frac{n}{4}-(1-(-1)^n)\frac{1}{8}-(1-(-1)^k)\frac{1}{4}}-1)+\ell(2^{\frac{n}{4}-(1-(-1)^n)\frac{1}{8}-(1-(-1)^k)\frac{1}{4}}+1)+2.\label{progresssecond decomp}
\end{align}
Inserting \eqref{progresssecond decomp} into \eqref{progressfirst decomp}, we get
\begin{align}
\ell(2^n-1)&\leq \ell(2^{\frac{n}{4}-(1-(-1)^n)\frac{1}{8}-(1-(-1)^k)\frac{1}{4}}-1)+\ell(2^{\frac{n}{4}-(1-(-1)^n)\frac{1}{8}-(1-(-1)^k)\frac{1}{4}}+1)+2\nonumber \\&+\ell(2^{\frac{n-(1-(-1)^n)\frac{1}{2}}{2}}+1)+2.\label{progresscombined}
\end{align}
We iterate the factorization to frequency $s$ to obtain 
\begin{align}
\ell(2^n-1)&\leq \ell(2^{\frac{n-(1-(-1)^n)\frac{1}{2}}{2}}+1)+2+\ell(2^{\frac{n}{4}-(1-(-1)^n)\frac{1}{8}-(1-(-1)^k)\frac{1}{4}}+1)+2\nonumber \\&+\cdots+\ell(2^{\frac{n}{2^s}-\xi(n,s)}-1)+\ell(2^{\frac{n}{2^s}-\xi(n,s)}+1)+2\label{progresslength bound}
\end{align}
where $0\leq \xi(n,s)<1$ for an integer $2\leq s:=s(n)$ fixed to be chosen later. For example, 
$$
\xi(n,1)=(1-(-1)^n)\frac{1}{4}<1
$$ 
and 
$$
\xi(n,2)=(1-(-1)^n)\frac{1}{8}+(1-(-1)^k)\frac{1}{4}<1
$$ 
with
$$
k:=\frac{n-(1-(-1)^n)\frac{1}{2}}{2}
$$ 
and so on. That is, $\xi(n,1):=\{\frac{n}{2}\}$ with $\xi(n,2)=\{\frac{1}{2}\lfloor \frac{n}{2}\rfloor\}$ and so on. The function $\xi(n,s)$ for values of $s\geq 3$ can be read from the exponents of the terms arising from the iteration process. We deduce from \eqref{progresslength bound}
\begin{align}
\ell(2^n-1)&\leq \sum \limits_{v=1}^{s}\frac{n}{2^v}+3s-\theta(n,s)+\ell(2^{\frac{n}{2^s}-\xi(n,s)}-1)\nonumber \\&=n(1-\frac{1}{2^{s}})+3s-\theta(n,s)+\ell(2^{\frac{n}{2^s}-\xi(n,s)}-1)\label{progressmajor inequality}
\end{align}
for some fixed $0\leq \theta(n,s):=\sum \limits_{j=1}^{s}\xi(n,j)$ and $2\leq s:=s(n)$ fixed, an integer to be chosen later. It is worth noting that 
$$
\theta(n,s):=\sum \limits_{j=1}^{s}\xi(n,j)=0
$$ 
if $n=2^r$ for some $r\in \mathbb{N}$ since $\xi(n,j)=0$ for each $1\leq j\leq s$ for all $n$ which are powers of $2$. It is also important to note that the $2s$ term is obtained by noting that there are at most $s$ terms with odd exponents under the iteration process and each term with odd exponent contributes $2$, and the other $s$ term comes from summing $1$ with frequency $s$ finding the total length of the short addition chains producing numbers of the form $2^v+1$. Now, we set $k=\frac{n}{2^s}-\xi(n,s)$ and construct the addition chain that leads to $2^{k}$ 
$$
1,2,2^2,\ldots, 2^{k-1},2^k
$$ 
with the corresponding partition sequence
\begin{align}
2=1+1,2+2=2^2,2^2+2^2=2^3\ldots,2^{k-1}=2^{k-2}+2^{k-2},2^k=2^{k-1}+2^{k-1}\nonumber
\end{align}
with $a_i=2^{i-2}=r_i$ for $2\leq i\leq k+1$, where $a_i$ and $r_i$ denote the determiner and regulator of the $i^{th}$ generator of the chain. We consider only the complete sub-addition chain 
\begin{align}
2=1+1,2+2=2^2,2^2+2^2=2^3,\ldots,2^{k-1}=2^{k-2}+2^{k-2}\nonumber
\end{align}
and extend this complete sub-addition chain by adjoining the sequence 
\begin{align}
2^{k-1}+2^{\lfloor\frac{k-1}{2}\rfloor},2^{k-1}+2^{\lfloor\frac{k-1}{2}\rfloor}+2^{\lfloor\frac{k-1}{2^2}\rfloor},\ldots,2^{k-1}+2^{\lfloor\frac{k-1}{2}\rfloor}+2^{\lfloor\frac{k-1}{2^2}\rfloor}+\cdots+2^1.\nonumber
\end{align}
Since $\xi(n,s)=0$ if $n=2^r$ and $0\leq \xi(n,s)<1$ if $n\neq 2^{r}$, we note that the adjoined sequence contributes at most 
\begin{align}
\left\lfloor\frac{\log k}{\log 2}\right\rfloor=\left\lfloor \frac{\log (\frac{n}{2^s}-\xi(n,s))}{\log 2}\right\rfloor=\left\lfloor \frac{\log n-s\log 2}{\log 2}\right\rfloor=\left\lfloor\frac{\log n}{\log 2}\right\rfloor-s\nonumber
\end{align}
terms to the original complete sub-addition chain. Due to the inequality
\begin{align}
2^{k-1}+2^{\lfloor \frac{k-1}{2}\rfloor}+2^{\lfloor \frac{k-1}{2^2}\rfloor}+\cdots+2^1&<\sum \limits_{i=1}^{k-1}2^i\nonumber \\&=2^k-2\nonumber
\end{align}
we insert terms into the sum
\begin{align}
2^{k-1}+2^{\lfloor \frac{k-1}{2}\rfloor}+2^{\lfloor \frac{k-1}{2^2}\rfloor}+\cdots+2^1\label{progressbroken road}
\end{align}
so that
\begin{align}
\sum \limits_{i=1}^{k-1}2^i=2^k-2.\nonumber 
\end{align}
We analyze the cost of filling in the missing terms of the underlying sum. We note that we have to insert $2^{k-2}+2^{k-3}+\cdots +2^{\lfloor \frac{k-1}{2}\rfloor+1}$ into \eqref{progressbroken road} and the number of terms adjoined is at most  
\begin{align}
k-2-\left\lfloor\frac{k-1}{2}\right\rfloor.\nonumber
\end{align}
The last term of the adjoined sequence is
\begin{align}
2^{k-1}+(2^{k-2}+2^{k-3}+\cdots +2^{\lfloor\frac{k-1}{2}\rfloor+1})+2^{\lfloor\frac{k-1}{2}\rfloor}+2^{\lfloor\frac{k-1}{2^2}\rfloor}+\cdots+2^1.\label{progressbroken road 1}
\end{align}
Again, we insert $2^{\lfloor \frac{k-1}{2}\rfloor-1}+\cdots+2^{\lfloor \frac{k-1}{2^2}\rfloor+1}$ into \eqref{progressbroken road 1} and the number of terms adjoined is at most 
\begin{align}
\left\lfloor\frac{k-1}{2}\right\rfloor-\left\lfloor\frac{k-1}{2^2}\right\rfloor-1\nonumber
\end{align}
The last term of the adjoined sequence is
\begin{align}
2^{k-1}+(2^{k-2}+2^{k-3}+\cdots+2^{\lfloor\frac{k-1}{2}\rfloor+1})+2^{\lfloor\frac{k-1}{2}\rfloor}+(2^{\lfloor \frac{k-1}{2}\rfloor-1}+\cdots+2^{\lfloor\frac{k-1}{2^2}\rfloor+1})+2^{\lfloor\frac{k-1}{2^2}\rfloor}+\nonumber \\ \cdots+2^1.\label{progressbroken road 2}
\end{align}
Iterating the process, we insert in the immediately previous term by inserting into \eqref{progressbroken road 2} and the number of terms adjoined is at most
\begin{align}
\left\lfloor\frac{k-1}{2^{j}}\right\rfloor-\left\lfloor\frac{k-1}{2^{j+1}}\right\rfloor-1\nonumber
\end{align}
for $j\leq \lfloor\frac{\log n}{\log 2}\rfloor-s$, since we fill at most $\lfloor\frac{\log k}{\log 2}\rfloor$ blocks with $k=\frac{n}{2^s}-\xi(n,s)$. The contribution of these new terms is at most
\begin{align}
k-1-\left\lfloor\frac{k-1}{2^{\lfloor\frac{\log k}{\log 2}\rfloor}}\right\rfloor-\left\lfloor\frac{\log k}{\log 2}\right\rfloor 
\end{align}
obtained by adding the numbers in the chain 
\begin{align}
k-1-\left\lfloor\frac{k-1}{2}\right\rfloor-1 \nonumber
\end{align}
\begin{align}
\left\lfloor\frac{k-1}{2}\right\rfloor-\left\lfloor\frac{k-1}{2^2}\right\rfloor-1\nonumber
\end{align}
\begin{align}
\vdots \vdots \vdots \vdots \vdots \vdots \vdots \vdots \vdots \vdots \vdots \vdots \nonumber
\end{align}
\begin{align}
\vdots \vdots \vdots \vdots \vdots \vdots \vdots \vdots \vdots \vdots \vdots \vdots \nonumber
\end{align}
\begin{align}
\left\lfloor\frac{k-1}{2^{\lfloor\frac{\log k}{\log 2}\rfloor}}\right\rfloor-\left\lfloor\frac{k-1}{2^{\lfloor\frac{\log k}{\log 2}\rfloor+1}}\right\rfloor-1.\nonumber
\end{align}
Taking a quick book-keeping, we deduce for $k=\frac{n}{2^s}-\xi(n,s)$ the inequality
\begin{align}
\delta(2^{k}-1)&\leq k+k-1-\left\lfloor\frac{k-1}{2^{\lfloor \frac{\log k}{\log 2}\rfloor+1}}\right\rfloor-\left\lfloor\frac{\log k}{\log 2}\right\rfloor+\left\lfloor\frac{\log n}{\log 2}\right\rfloor-s\nonumber \\&\leq \frac{n}{2^{s-1}}-1-\left\lfloor\frac{\frac{n}{2^s}-\xi(n,s)-1}{2^{\lfloor \frac{\log n}{\log 2}\rfloor+1-s}}\right\rfloor-\left\lfloor\frac{\log n}{\log 2}\right\rfloor+s +\left\lfloor\frac{\log n}{\log 2}\right\rfloor-s\nonumber \\&=\frac{n}{2^{s-1}}-1-\left \lfloor \frac{\frac{n}{2^s}-\xi(n,s)-1}{2^{\lfloor \frac{\log n}{\log 2}\rfloor+1-s}}\right \rfloor.\label{progresslast analysis}
\end{align}
Plugging the inequality \eqref{progresslast analysis} into the inequalities in \eqref{progressmajor inequality} and noting that $\ell(\cdot)\leq \delta(\cdot)$, we deduce
\begin{align}
\ell(2^n-1)&\leq \sum \limits_{v=1}^{s}\frac{n}{2^v}+3s-\theta(n,s)+\ell(2^{\frac{n}{2^s}-\xi(n,s)}-1)\nonumber \\&=n(1-\frac{1}{2^{s}})+\frac{n}{2^{s-1}}-1+3s-\theta(n,s)-\left \lfloor \frac{\frac{n}{2^s}-\xi(n,s)-1}{2^{\lfloor \frac{\log n}{\log 2}\rfloor+1-s}}\right \rfloor \\&=n-1+\frac{n}{2^s}+3s-\theta(n,s)-\left \lfloor \frac{\frac{n}{2^s}-\xi(n,s)-1}{2^{\lfloor \frac{\log n}{\log 2}\rfloor+1-s}}\right \rfloor.\nonumber
\end{align}
Taking $2\leq s:=s(n)$ such that $s=\lfloor\frac{\log n}{\log 2}\rfloor$, we get 
$$
\left \lfloor \frac{\frac{n}{2^s}-\xi(n,s)-1}{2^{\lfloor \frac{\log n}{\log 2}\rfloor+1-s}}\right\rfloor=0
$$ 
and obtain
$$\ell(2^n-1)\leq n-1-\theta \left(n,\left\lfloor\frac{\log n}{\log 2}\right\rfloor\right)+2+3\left\lfloor\frac{\log n}{\log 2}\right\rfloor
$$ 
for $\theta(n,\lfloor \frac{\log n}{\log 2}\rfloor):=\sum \limits_{j=1}^{\lfloor \frac{\log n}{\log 2}\rfloor}\xi(n,j)$ with $n>4$. 
\end{proof}
\bigskip

In the sequel, we prove a much more general inequality. This inequality would be useful for our subsequent investigations of the Scholz conjecture, which comes from the general inequality developed in the proof of Theorem \ref{exceptional scholz conjecture}.

\begin{proposition}
There exists an $s$ with $\lfloor\frac{\log n}{\log 2}\rfloor \geq s:=s(n)\geq 2$ such that 
$$
\ell(2^n-1)\leq n-1+\left(3s-\sum \limits_{j=1}^{s}\{\frac{n}{2^j}\}\right)+\ell \left(\left\lfloor\frac{n}{2^s}\right\rfloor\right)
$$ 
for all $n\geq 2$, where $\{\cdot\}$ and $\lfloor\cdot\rfloor$ denote the fractional part and the integer part of any real number $\cdot$.
\end{proposition}

\begin{proof}
We recall the general inequality developed in the proof of Theorem \ref{exceptional scholz conjecture} 
\begin{align}
\ell(2^n-1)&\leq \sum \limits_{v=1}^{s}\frac{n}{2^v}+3s-\theta(n,s)+\ell(2^{\frac{n}{2^s}-\xi(n,s)}-1)\nonumber \\&=n(1-\frac{1}{2^{s}})+3s-\theta(n,s)+\ell(2^{\frac{n}{2^s}-\xi(n,s)}-1)\label{corollarymajor inequality}
\end{align}
for some $0\leq \theta(n,s):=\sum \limits_{j=1}^{s}\xi(n,j)$ and fixed $s$ with $2\leq s:=s(n)$. We check $\xi(n,j)=\{\frac{n}{2^j}\}$. It is known that the Scholz conjecture is true for all exponents up to $5784688$ and it is also true infinitely often so that we can choose an integer $s\geq 2$ such that the exponent
$$
2^{\frac{n}{2^s}-\xi(n,s)}-1=2^{\frac{n}{2^s}-\{\frac{n}{2^s}\}}-1
$$ 
satisfies the Scholz conjecture. We can choose $s:=s(n)\geq 2$ such that the inequality
\begin{align}
\ell(2^{\frac{n}{2^s}-\xi(n,s)}-1)&=\ell(2^{\lfloor\frac{n}{2^s}\rfloor}-1)\leq \left\lfloor\frac{n}{2^s}\right\rfloor-1+\ell \left(\left\lfloor\frac{n}{2^s}\right\rfloor\right)\nonumber \\&=\frac{n}{2^s}-\{\frac{n}{2^s}\}-1+\ell \left(\left\lfloor\frac{n}{2^s}\right\rfloor\right)\label{corollaryinclude}
\end{align}
so that by plugging \eqref{corollaryinclude} into \eqref{corollarymajor inequality}, the claimed inequality follows as a consequence.
\end{proof}

\section{Carry and carry analysis}

We devote this section to the study of the concept of carry and its number theoretic properties. It turns out that this notion plays an important role in controlling the length of an addition for numbers of the form $2^n-1$. Short addition chains with small carries almost satisfy the Scholz conjecture.

\begin{definition}
We write 
$$
2^n-1=(2^{\lfloor \frac{n}{2}\rfloor}-1)(2^{\lfloor \frac{n}{2}\rfloor }+1)+\frac{(1-(-1)^n)}{2}(2^{n-(1-(-1)^n)\frac{1}{2}})
$$ 
for $n\geq 2$. The non-zero remainder 
$$
\eta(2^n-1):=\frac{(1-(-1)^n)}{2}(2^{n-(1-(-1)^n)\frac{1}{2}})
$$ 
is the level one carry of $2^n-1$. We say that $2^n-1$ is free of level one carries if $\eta(2^n-1)=0$. Setting 
$$
m=\left\lfloor\frac{n}{2}\right\rfloor
$$ 
we write
$$
2^m-1=(2^{\lfloor\frac{m}{2}\rfloor}-1)(2^{\lfloor\frac{m}{2}\rfloor}+1)+\frac{(1-(-1)^m)}{2}(2^{m-(1-(-1)^m)\frac{1}{2}})
$$ 
and denote the carry with 
$$
\eta(2^m-1)=\frac{(1-(-1)^m)}{2}(2^{m-(1-(-1)^m)\frac{1}{2}})
$$ 
and say that it is the level two carry of $2^n-1$ if $\eta(2^m-1)\neq 0.$ In general, we denote the level $k$ carry of $2^n-1$ by 
$$
\eta(2^r-1)=\frac{(1-(-1)^r)}{2}(2^{r-(1-(-1)^r)\frac{1}{2}})
$$ 
with 
$$
r=\left\lfloor\frac{n}{2^k}\right\rfloor.
$$
We say that $2^n-1$ is free from the level $k$ carry if $\eta(2^r-1)=0.$ The number of non-zero levels of carry of $2^n-1$ for all $1\leq k\leq \lfloor \frac{\log n}{\log 2}\rfloor$ is the \emph{degree} of carry of $2^n-1$.
\end{definition}
\bigskip

\begin{proposition}
The number $2^n-1$ \quad ($n\geq 2$) is free from the level one carry if and only if $n\equiv 0\pmod 2.$
\end{proposition}

\begin{proof}
Suppose that $2^n-1$ is free from the level one carry. This implies that 
$$
\eta(2^n-1)=\frac{(1-(-1)^n)}{2}(2^{n-(1-(-1)^n)\frac{1}{2}})=0.
$$ 
This is only possible with $(1-(-1)^n)=0$ and when $n\equiv 0\pmod 2.$ Conversely, suppose that $n\equiv 0\pmod 2$. This implies $\frac{n}{2}\in \mathbb{N}$ and we can write 
$$
2^n-1=(2^{\frac{n}{2}}-1)(2^{\frac{n}{2}}+1)
$$ 
and deduce
$$
\eta(2^n-1)=0.
$$
\end{proof}
\bigskip

Integers of the form $2^n-1$ with high degrees of carry serve as an obstruction to achieving the inequality 
$$
\ell(2^n-1)\leq n-1+\ell(n)
$$ 
using our current method. In the best case, avoiding them can lead to progress on the conjecture using the current method but only for a specialized set of integers of the form $2^n-1$ with low degrees of carry. It turns out that the nature of the exponents in large part characterizes integers with high degree~(resp.~low degree) carries. The encounter of integers of the form $2^n-1$ with exponents giving rise to high degree carries can be controlled in a way to minimize the corresponding length of the addition chain. At the moment, we can obtain a chain of small length for numbers $2^n-1$ with exponents that give rise to low degree carries.

\begin{theorem}\label{exceptional scholz conjecture 1}
If $2^n-1$ has carry of degree at most 
$$
\kappa(2^n-1)=\frac{1}{2(1+c)}\left\lfloor \frac{\log n}{\log 2}\right\rfloor-1
$$ 
for a fixed $c>0$, then
$$
\ell(2^n-1)\leq n-1+\left(1+\frac{1}{1+c}\right)\left\lfloor\frac{\log n}{\log 2}\right\rfloor
$$ 
for all $n\in \mathbb{N}$ with $n\geq 4$, where $\ell(\cdot)$ denotes the length of the shortest addition chain that leads to $\cdot$.
\end{theorem}

\begin{proof}
  For a fixed $c>0$ assume that $2^n-1$ has at most 
  $$
  \frac{1}{2(1+c)}\left\lfloor\frac{\log n}{\log 2}\right\rfloor-1
  $$ 
  degrees of carry. We write 
  $$
  2^n-1=(2^{\lfloor\frac{n}{2}\rfloor}-1)(2^{\lfloor \frac{n}{2}\rfloor }+1)+\eta(2^n-1)
  $$ 
  where 
  $$
  \eta(2^n-1):=\frac{(1-(-1)^n)}{2}(2^{n-(1-(-1)^n)\frac{1}{2}})
  $$ 
  is the level one carry of $2^n-1.$ We can recover the general factorization of $2^n-1$ from this identity according to the parity of the exponent $n$. In particular, if $n\equiv 0\pmod 2$, then
  $$
  2^n-1=(2^{\frac{n}{2}}-1)(2^{\frac{n}{2}}+1)
  $$ 
  and 
  $$
  2^n-1=(2^{\frac{n-1}{2}}-1)(2^{\frac{n-1}{2}}+1)+2^{n-1}
  $$ 
  if $n\equiv 1\pmod 2$. Combining both cases, we obtain
  $$
  \ell(2^n-1)\leq \ell \left((2^{\lfloor\frac{n}{2}\rfloor}-1)(2^{\lfloor\frac{n}{2}\rfloor}+1)\right)+\eta(2^n-1).$$ Applying the lemma \ref{Alfred Braurer}, we further deduce
\begin{align}
\ell(2^n-1)\leq \ell(2^{\lfloor\frac{n}{2}\rfloor }-1)+\ell(2^{\lfloor\frac{n}{2}\rfloor}+1)+\eta(2^n-1)\label{firstcanalyfirst decomp}
\end{align}
We set $\lfloor \frac{n}{2}\rfloor=k$ in \eqref{firstcanalyfirst decomp} and write 
$$
2^k-1=(2^{\lfloor \frac{k}{2}\rfloor }-1)(2^{\lfloor \frac{k}{2}\rfloor }+1)+\eta(2^k-1)
$$ 
where 
$$
\eta(2^k-1)=\frac{(1-(-1)^k)}{2}(2^{k-(1-(-1)^k)\frac{1}{2}})
$$ 
is the carry of $2^k-1$. We can recover the general factorization of $2^k-1$ from this identity according to the parity of the exponent $k$. In particular, if $k\equiv 0\pmod 2$, then 
$$
2^k-1=(2^{\frac{k}{2}}-1)(2^{\frac{k}{2}}+1)
$$ 
and 
$$
2^k-1=(2^{\frac{k-1}{2}}-1)(2^{\frac{k-1}{2}}+1)+2^{k-1}
$$ 
if $k\equiv 1\pmod 2$. Combining both cases, we obtain
$$
\ell(2^k-1)\leq \ell((2^{\lfloor\frac{k}{2}\rfloor}-1)(2^{\lfloor \frac{k}{2}\rfloor}+1))+\eta(2^k-1).
$$
Applying the lemma \ref{Alfred Braurer}, we deduce
\begin{align}
\ell(2^k-1)&\leq \ell(2^{\lfloor\frac{k}{2}\rfloor }-1)+\ell(2^{\lfloor\frac{k}{2}\rfloor}+1)+\eta(2^k-1)\nonumber \\&=\ell(2^{\lfloor\frac{1}{2}\lfloor\frac{n}{2}\rfloor \rfloor}-1)+\ell(2^{\lfloor\frac{1}{2}\lfloor\frac{n}{2}\rfloor \rfloor}+1)+\eta(2^{\lfloor \frac{n}{2}\rfloor}-1)\label{firstcanalysecond decomp}
\end{align}
so that by inserting \eqref{firstcanalysecond decomp} into \eqref{firstcanalyfirst decomp}, we get
\begin{align}
\ell(2^n-1)&\leq \ell(2^{\lfloor\frac{1}{2}\lfloor\frac{n}{2}\rfloor\rfloor}-1)+\ell(2^{\lfloor\frac{1}{2}\lfloor\frac{n}{2}\rfloor\rfloor}+1)+\eta(2^{\lfloor \frac{n}{2}\rfloor}-1)\nonumber \\&+\ell(2^{\lfloor \frac{n}{2}\rfloor }+1)+\eta(2^n-1).\label{firstcanalycombined}
\end{align}
We iterate the factorization to frequency $s$ to obtain 
\begin{align}
\ell(2^n-1)&\leq \ell(2^{\lfloor\frac{n}{2}\rfloor }+1)+\eta(2^n-1)+\ell(2^{\lfloor\frac{1}{2}\lfloor\frac{n}{2}\rfloor \rfloor}-1)+\ell(2^{\lfloor\frac{1}{2}\lfloor\frac{n}{2}\rfloor \rfloor}+1)+\eta(2^{\lfloor \frac{n}{2}\rfloor}-1)\nonumber \\&+\cdots+\ell(2^{\frac{n}{2^s}-\xi(n,s)}-1)+\ell(2^{\frac{n}{2^s}-\xi(n,s)}+1)+\eta(2^{\lfloor\frac{n}{2^{s-1}}\rfloor}-1)\label{firstcanalylength bound}
\end{align}
where $0\leq \xi(n,s)<1$ for an integer $2\leq s:=s(n)$ fixed to be chosen later. For example
$$
\xi(n,1)=(1-(-1)^n)\frac{1}{4}<1
$$ 
and 
$$
\xi(n,2)=(1-(-1)^n)\frac{1}{8}+(1-(-1)^k)\frac{1}{4}<1
$$ 
with 
$$
k:=\left\lfloor\frac{n}{2}\right\rfloor
$$ 
and so on. That is, $\xi(n,1):=\{\frac{n}{2}\}$ with $\xi(n,2)=\{\frac{1}{2}\lfloor\frac{n}{2}\rfloor\}$ and so on. The function $\xi(n,s)$ for values of $s\geq 3$ can be read from exponents of the terms arising from the iteration process. We deduce from \eqref{firstcanalylength bound} the inequality 
\begin{align}
\ell(2^n-1)&\leq \sum\limits_{v=1}^{s}\frac{n}{2^v}+s+2\sum \limits_{j=1}^{s}\sum\limits_{\substack{\eta(2^m-1)\neq 0\\m=\lfloor\frac{n}{2^{j-1}}\rfloor}}1-\theta(n,s)+\ell(2^{\frac{n}{2^s}-\xi(n,s)}-1)\nonumber \\&=n(1-\frac{1}{2^{s}})+s+2\sum \limits_{j=1}^{s}\sum \limits_{\substack{\eta(2^m-1)\neq 0\\m=\lfloor \frac{n}{2^{j-1}}\rfloor }}1-\theta(n,s)+\ell(2^{\frac{n}{2^s}-\xi(n,s)}-1)\label{firstcanalymajor inequality}
\end{align}
where the term 
$$
\sum \limits_{j=1}^{s}\sum \limits_{\substack{\kappa(2^m-1)\neq 0\\m=\lfloor \frac{n}{2^{j-1}}\rfloor }}1
$$ 
counts the number of all carry of $2^n-1$ up to level $s$ and $0\leq \theta(n,s):=\sum \limits_{j=1}^{s}\xi(n,j)$ and $2\leq s:=s(n)$ is a fixed integer to be chosen later. We note that 
$$
\theta(n,s):=\sum \limits_{j=1}^{s}\xi(n,j)=0
$$ 
if $n=2^r$ for some $r\in \mathbb{N}$, since $\xi(n,j)=0$ for each $1\leq j\leq s$ for all $n$ which are powers of $2$. The $2s$ term is obtained by noting that there are at most $s$ terms with odd exponents under the iteration process and each term with odd exponent contributes $2$, and the other $s$ term comes from summing $1$ with frequency $s$ finding the total length of the short addition chains that lead numbers of the form $2^v+1$. Now, we set $k=\frac{n}{2^s}-\xi(n,s)$ and construct the addition chain producing $2^{k}$ 
$$
1,2,2^2,\ldots,2^{k-1},2^k
$$ 
with the corresponding partition sequence
\begin{align}
2=1+1,2+2=2^2,2^2+2^2=2^3\ldots,2^{k-1}=2^{k-2}+2^{k-2},2^k=2^{k-1}+2^{k-1}\nonumber
\end{align}
with $a_i=2^{i-2}=r_i$ for $2\leq i\leq k+1$, where $a_i$ and $r_i$ denote the determiner and regulator of the $i^{th}$ generator of the chain. We consider only the complete sub-addition chain 
\begin{align}
2=1+1,2+2=2^2,2^2+2^2=2^3,\ldots,2^{k-1}=2^{k-2}+2^{k-2}\nonumber
\end{align}
and extend this complete sub-addition chain by adjoining the sequence 
\begin{align}
2^{k-1}+2^{\lfloor\frac{k-1}{2}\rfloor},2^{k-1}+2^{\lfloor\frac{k-1}{2}\rfloor}+2^{\lfloor\frac{k-1}{2^2}\rfloor},\ldots,2^{k-1}+2^{\lfloor\frac{k-1}{2}\rfloor}+2^{\lfloor\frac{k-1}{2^2}\rfloor}+\cdots+2^1.\nonumber
\end{align}
Since $\xi(n,s)=0$ if $n=2^r$ and $0\leq \xi(n,s)<1$ if $n\neq 2^{r}$, we note that the adjoined sequence contributes at most 
\begin{align}
\left\lfloor\frac{\log k}{\log 2}\right\rfloor=\left\lfloor\frac{\log (\frac{n}{2^s}-\xi(n,s))}{\log 2}\right\rfloor=\left\lfloor\frac{\log n-s\log 2}{\log 2}\right\rfloor=\left\lfloor\frac{\log n}{\log 2}\right\rfloor-s\nonumber
\end{align}
terms to the original complete sub-addition chain. Due to the inequality
\begin{align}
2^{k-1}+2^{\lfloor\frac{k-1}{2}\rfloor}+2^{\lfloor\frac{k-1}{2^2}\rfloor}+\cdots+2^1&<\sum\limits_{i=1}^{k-1}2^i\nonumber \\&=2^k-2\nonumber
\end{align}
we insert terms into the sum
\begin{align}
2^{k-1}+2^{\lfloor \frac{k-1}{2}\rfloor}+2^{\lfloor \frac{k-1}{2^2}\rfloor}+\cdots+2^1\label{firstcanalybroken road}
\end{align}
so that
\begin{align}
\sum \limits_{i=1}^{k-1}2^i=2^k-2.\nonumber 
\end{align}
We analyze the cost of filling in the missing terms of the underlying sum. We insert $2^{k-2}+2^{k-3}+\cdots +2^{\lfloor \frac{k-1}{2}\rfloor+1}$ into \eqref{firstcanalybroken road} and the number of terms adjoined is at most   
\begin{align}
k-2-\left\lfloor\frac{k-1}{2}\right\rfloor.\nonumber
\end{align}
The last term of the adjoined sequence is given by 
\begin{align}
2^{k-1}+(2^{k-2}+2^{k-3}+\cdots +2^{\lfloor\frac{k-1}{2}\rfloor+1})+2^{\lfloor\frac{k-1}{2}\rfloor}+2^{\lfloor\frac{k-1}{2^2}\rfloor}+\cdots+2^1.\label{firstcanalybroken road 1}
\end{align}
Again, we insert $2^{\lfloor\frac{k-1}{2}\rfloor-1}+\cdots+2^{\lfloor\frac{k-1}{2^2}\rfloor+1}$ into \eqref{firstcanalybroken road 1} and the number of terms adjoined is at most 
\begin{align}
\left\lfloor\frac{k-1}{2}\right\rfloor-\left\lfloor\frac{k-1}{2^2}\right\rfloor-1.\nonumber
\end{align}
The last term of the adjoined sequence is
\begin{align}
2^{k-1}+(2^{k-2}+2^{k-3}+\cdots +2^{\lfloor\frac{k-1}{2}\rfloor+1})+2^{\lfloor\frac{k-1}{2}\rfloor}+(2^{\lfloor\frac{k-1}{2}\rfloor-1}+\cdots+2^{\lfloor\frac{k-1}{2^2}\rfloor+1})+2^{\lfloor\frac{k-1}{2^2}\rfloor}+\nonumber \\ \cdots+2^1.\label{firstcanalybroken road 2}
\end{align}
Iterating the process, we insert into the immediately previous term by inserting into \eqref{firstcanalybroken road 2} and the number of terms of adjoined is at most
\begin{align}
\left\lfloor\frac{k-1}{2^{j}}\right\rfloor-\left\lfloor\frac{k-1}{2^{j+1}}\right\rfloor-1\nonumber
\end{align}
for $j\leq \lfloor\frac{\log n}{\log 2}\rfloor-s$, since we are filling at most $\lfloor\frac{\log k}{\log 2}\rfloor$ blocks with $k=\frac{n}{2^s}-\xi(n,s)$. The contribution of these new terms is at most
\begin{align}
k-1-\left\lfloor\frac{k-1}{2^{\lfloor\frac{\log k}{\log 2}\rfloor}}\right\rfloor-\left\lfloor\frac{\log k}{\log 2}\right\rfloor 
\end{align}
obtained by adding the numbers in the chain 
\begin{align}
k-1-\left\lfloor\frac{k-1}{2}\right\rfloor-1 \nonumber
\end{align}
\begin{align}
\left\lfloor\frac{k-1}{2}\right\rfloor-\left\lfloor\frac{k-1}{2^2}\right\rfloor-1\nonumber
\end{align}
\begin{align}
\vdots \vdots \vdots \vdots \vdots \vdots \vdots \vdots \vdots \vdots \vdots \vdots \nonumber
\end{align}
\begin{align}
\vdots \vdots \vdots \vdots \vdots \vdots \vdots \vdots \vdots \vdots \vdots \vdots \nonumber
\end{align}
\begin{align}
\left\lfloor\frac{k-1}{2^{\lfloor\frac{\log k}{\log 2}\rfloor}}\right\rfloor-\left\lfloor\frac{k-1}{2^{\lfloor \frac{\log k}{\log 2}\rfloor+1}}\right\rfloor-1.\nonumber
\end{align}
By a quick book-keeping, we deduce with $k=\frac{n}{2^s}-\xi(n,s)$ for the length
\begin{align}
\delta(2^{k}-1)&\leq k+k-1-\left\lfloor\frac{k-1}{2^{\lfloor \frac{\log k}{\log 2}\rfloor+1}}\right\rfloor-\lfloor\frac{\log k}{\log 2}\rfloor+\left\lfloor\frac{\log n}{\log 2}\right\rfloor-s\nonumber \\&\leq \frac{n}{2^{s-1}}-1-\left\lfloor\frac{\frac{n}{2^s}-\xi(n,s)-1}{2^{\lfloor\frac{\log n}{\log 2}\rfloor+1-s}}\right\rfloor-\left\lfloor\frac{\log n}{\log 2}\right\rfloor+s +\left\lfloor\frac{\log n}{\log 2}\right\rfloor-s\nonumber \\&=\frac{n}{2^{s-1}}-1-\left\lfloor\frac{\frac{n}{2^s}-\xi(n,s)-1}{2^{\lfloor\frac{\log n}{\log 2}\rfloor+1-s}}\right \rfloor.\label{firstcanalylast analysis}
\end{align}
Plugging the inequality \eqref{firstcanalylast analysis} into the inequalities in \eqref{firstcanalymajor inequality} and noting that $\ell(\cdot)\leq \delta(\cdot)$, we obtain
\begin{align}
\ell(2^n-1)&\leq \sum\limits_{v=1}^{s}\frac{n}{2^v}+s+2\sum \limits_{j=1}^{s}\sum\limits_{\substack{\eta(2^m-1)\neq 0\\m=\lfloor \frac{n}{2^{j-1}}\rfloor}}1-\theta(n,s)+\ell(2^{\frac{n}{2^s}-\xi(n,s)}-1)\nonumber \\&=n(1-\frac{1}{2^{s}})+\frac{n}{2^{s-1}}-1+s+2\sum\limits_{j=1}^{s}\sum \limits_{\substack{\eta(2^m-1)\neq 0\\m=\lfloor\frac{n}{2^{j-1}}\rfloor}}1-\theta(n,s)-\left\lfloor\frac{\frac{n}{2^s}-\xi(n,s)-1}{2^{\lfloor\frac{\log n}{\log 2}\rfloor+1-s}}\right \rfloor \nonumber \\&=n-1+\frac{n}{2^s}+s+2\sum \limits_{j=1}^{s}\sum \limits_{\substack{\eta(2^m-1)\neq 0\\m=\lfloor \frac{n}{2^{j-1}}\rfloor}}1-\theta(n,s)-\left \lfloor \frac{\frac{n}{2^s}-\xi(n,s)-1}{2^{\lfloor\frac{\log n}{\log 2}\rfloor+1-s}}\right\rfloor\nonumber
\end{align}
where
$$
\sum\limits_{j=1}^{s}\sum\limits_{\substack{\eta(2^m-1)\neq 0\\m=\lfloor \frac{n}{2^{j-1}}\rfloor}}1
$$ 
counts the number of non-zero carries up to the $s$ level for the number $2^n-1$. Taking $2\leq s:=s(n)$ such that $s= \lfloor \frac{\log n}{\log 2}\rfloor$ which is the maximum frequency of the iteration, we get
$$
\left\lfloor\frac{\frac{n}{2^s}-\xi(n,s)-1}{2^{\lfloor\frac{\log n}{\log 2}\rfloor+1-s}}\right\rfloor=0
$$ 
and obtain 
$$
\sum\limits_{j=1}^{s}\sum\limits_{\substack{\eta(2^m-1)\neq 0\\m=\lfloor\frac{n}{2^{j-1}}\rfloor }}1\leq \frac{1}{2(1+c)}\lfloor\frac{\log n}{\log 2}\rfloor-1
$$ 
and the inequality 
$$
\ell(2^n-1)\leq n-1-\theta \left(n,\left\lfloor\frac{\log n}{\log 2}\right\rfloor\right)+\left\lfloor\frac{\log n}{\log 2}\right\rfloor+2+\frac{1}{(1+c)}\left\lfloor\frac{\log n}{\log 2}\right\rfloor-2
$$ 
for $\theta(n,\lfloor\frac{\log n}{\log 2}\rfloor):=\sum \limits_{j=1}^{\lfloor \frac{\log n}{\log 2}\rfloor}\xi(n,j)>0$ with $n\geq 4$ and the claimed inequality follows as a consequence. 
\end{proof}
\bigskip

Now, we show that numbers of the form $2^n-1$ with low degree carry almost satisfy the Scholz conjecture. 

\begin{theorem}\label{almost scholz}
If $2^n-1$ has carry of degree at most 
$$
\kappa(2^n-1):=\left(\frac{1}{1+\log n}\right)\left\lfloor\frac{\log n}{\log 2}\right\rfloor-1
$$ 
then
$$
\ell(2^n-1)\leq n-1+\left(1+\frac{2}{1+\log n}\right)\left\lfloor\frac{\log n}{\log 2}\right\rfloor
$$ 
for all $n\in \mathbb{N}$ with $n\geq 4$, where $\ell(\cdot)$ denotes the length of the shortest addition chain producing $\cdot$.
\end{theorem}

\begin{proof}
Suppose that $2^n-1$ has at most 
$$
\frac{1}{(1+\log n)}\lfloor\frac{\log n}{\log 2}\rfloor-1
$$ 
degrees of carry. We write 
$$
2^n-1=(2^{\lfloor\frac{n}{2}\rfloor}-1)(2^{\lfloor\frac{n}{2}\rfloor }+1)+\eta(2^n-1)
$$ 
where 
$$
\eta(2^n-1):=\frac{(1-(-1)^n)}{2}(2^{n-(1-(-1)^n)\frac{1}{2}})
$$ 
is the level one carry of $2^n-1.$ We can recover the general factorization of $2^n-1$ from this identity according to the parity of the exponent $n$. In particular, if $n\equiv 0\pmod 2$, then
$$
2^n-1=(2^{\frac{n}{2}}-1)(2^{\frac{n}{2}}+1)
$$ 
and 
$$
2^n-1=(2^{\frac{n-1}{2}}-1)(2^{\frac{n-1}{2}}+1)+2^{n-1}
$$ 
if $n\equiv 1\pmod 2$. Combining both cases, we obtain
$$
\ell(2^n-1)\leq \ell((2^{\lfloor\frac{n}{2}\rfloor}-1)(2^{\lfloor \frac{n}{2}\rfloor}+1))+\eta(2^n-1).
$$ 
Applying the lemma \ref{Alfred Braurer}, we further obtain
\begin{align}
\ell(2^n-1)\leq \ell(2^{\lfloor \frac{n}{2}\rfloor }-1)+\ell(2^{\lfloor \frac{n}{2}\rfloor }+1)+\eta(2^n-1).\label{secondcanalyfirst decomp}
\end{align}
We set $\lfloor \frac{n}{2}\rfloor=k$ in \eqref{secondcanalyfirst decomp} and write
$$
2^k-1=(2^{\lfloor\frac{k}{2}\rfloor}-1)(2^{\lfloor\frac{k}{2}\rfloor}+1)+\eta(2^k-1)
$$ 
where 
$$
\eta(2^k-1)=\frac{(1-(-1)^k)}{2}(2^{k-(1-(-1)^k)\frac{1}{2}})
$$ 
is the carry of $2^k-1$. We can recover the general factorization of $2^k-1$ from this identity according to the parity of the exponent $k$. In particular, if $k\equiv 0\pmod 2$, then 
$$
2^k-1=(2^{\frac{k}{2}}-1)(2^{\frac{k}{2}}+1)
$$ 
and
$$
2^k-1=(2^{\frac{k-1}{2}}-1)(2^{\frac{k-1}{2}}+1)+2^{k-1}
$$ 
if $k\equiv 1\pmod 2$. Combining both cases, we get
$$
\ell(2^k-1)\leq \ell((2^{\lfloor\frac{k}{2}\rfloor}-1)(2^{\lfloor \frac{k}{2}\rfloor }+1))+\eta(2^k-1).
$$ 
Applying the lemma \ref{Alfred Braurer}, we further obtain
\begin{align}
\ell(2^k-1)&\leq \ell(2^{\lfloor\frac{k}{2}\rfloor }-1)+\ell(2^{\lfloor \frac{k}{2}\rfloor }+1)+\eta(2^k-1)\nonumber \\&=\ell(2^{\lfloor \frac{1}{2}\lfloor\frac{n}{2}\rfloor \rfloor}-1)+\ell(2^{\lfloor\frac{1}{2}\lfloor\frac{n}{2}\rfloor \rfloor}+1)+\eta(2^{\lfloor\frac{n}{2}\rfloor}-1)\label{secondcanalysecond decomp}
\end{align}
so that by inserting \eqref{secondcanalysecond decomp} into \eqref{secondcanalyfirst decomp}, we get
\begin{align}
\ell(2^n-1)&\leq \ell(2^{\lfloor\frac{1}{2}\lfloor\frac{n}{2}\rfloor\rfloor}-1)+\ell(2^{\lfloor \frac{1}{2}\lfloor\frac{n}{2}\rfloor\rfloor}+1)+\eta(2^{\lfloor\frac{n}{2}\rfloor}-1)\nonumber \\&+\ell(2^{\lfloor\frac{n}{2}\rfloor }+1)+\eta(2^n-1).\label{secondcanalycombined}
\end{align}
We iterate the factorization to frequency $s$ to obtain 
\begin{align}
\ell(2^n-1)&\leq \ell(2^{\lfloor\frac{n}{2}\rfloor }+1)+\eta(2^n-1)+\ell(2^{\lfloor\frac{1}{2}\lfloor\frac{n}{2}\rfloor \rfloor}-1)+\ell(2^{\lfloor\frac{1}{2}\lfloor\frac{n}{2}\rfloor \rfloor}+1)+\eta(2^{\lfloor\frac{n}{2}\rfloor}-1)\nonumber \\&+\cdots+\ell(2^{\frac{n}{2^s}-\xi(n,s)}-1)+\ell(2^{\frac{n}{2^s}-\xi(n,s)}+1)+\eta(2^{\lfloor \frac{n}{2^{s-1}}\rfloor}-1)\label{secondcanalylength bound}
\end{align}
where $0\leq \xi(n,s)<1$ is a fixed integer $2\leq s:=s(n)$ to be chosen later. For example
$$
\xi(n,1)=(1-(-1)^n)\frac{1}{4}<1
$$ 
and 
$$
\xi(n,2)=(1-(-1)^n)\frac{1}{8}+(1-(-1)^k)\frac{1}{4}<1
$$ 
with 
$$
k:=\left\lfloor\frac{n}{2}\right\rfloor
$$ 
and so on. That is, $\xi(n,1):=\{\frac{n}{2}\}$ with $\xi(n,2)=\{\frac{1}{2}\lfloor\frac{n}{2}\rfloor\}$ and so on. The function $\xi(n,s)$ for values of $s\geq 3$ can be read from exponents of the terms arising from the iteration process. We deduce from \eqref{secondcanalylength bound} the inequality 
\begin{align}
\ell(2^n-1)&\leq \sum\limits_{v=1}^{s}\frac{n}{2^v}+s+2\sum \limits_{j=1}^{s}\sum\limits_{\substack{\eta(2^m-1)\neq 0\\m=\lfloor \frac{n}{2^{j-1}}\rfloor }}1-\theta(n,s)+\ell(2^{\frac{n}{2^s}-\xi(n,s)}-1)\nonumber \\&=n(1-\frac{1}{2^{s}})+s+2\sum \limits_{j=1}^{s}\sum \limits_{\substack{\eta(2^m-1)\neq 0\\m=\lfloor\frac{n}{2^{j-1}}\rfloor }}1-\theta(n,s)+\ell(2^{\frac{n}{2^s}-\xi(n,s)}-1)\label{secondcanalymajor inequality}
\end{align}
where the term 
$$
\sum\limits_{j=1}^{s}\sum\limits_{\substack{\kappa(2^m-1)\neq 0\\m=\lfloor\frac{n}{2^{j-1}}\rfloor}}1
$$ 
counts the number of all carry of $2^n-1$ up to level $s$ and $0\leq \theta(n,s):=\sum\limits_{j=1}^{s}\xi(n,j)$ and $2\leq s:=s(n)$ fixed, an integer to be chosen later. We note that 
$$
\theta(n,s):=\sum\limits_{j=1}^{s}\xi(n,j)=0
$$ 
if $n=2^r$ for some $r\in \mathbb{N}$, since $\xi(n,j)=0$ for each $1\leq j\leq s$ for all $n$ that are powers of $2$. The $2s$ term is obtained by noting that there are at most $s$ terms with odd exponents under the iteration process and each term with odd exponent contributes $2$, and the other $s$ term comes from summing $1$ with frequency $s$ finding the total length of the short addition chains producing numbers of the form $2^v+1$. Now, we set $k=\frac{n}{2^s}-\xi(n,s)$ and construct the addition chain producing $2^{k}$ 
$$
1,2,2^2,\ldots,2^{k-1},2^k
$$ 
with corresponding partition sequence
\begin{align}
2=1+1,2+2=2^2,2^2+2^2=2^3\ldots,2^{k-1}=2^{k-2}+2^{k-2},2^k=2^{k-1}+2^{k-1}\nonumber
\end{align}
with $a_i=2^{i-2}=r_i$ for $2\leq i\leq k+1$, where $a_i$ and $r_i$ denote the determiner and regulator of the $i^{th}$ generator of the chain. We consider only the complete sub-addition chain 
\begin{align}
2=1+1,2+2=2^2,2^2+2^2=2^3,\ldots,2^{k-1}=2^{k-2}+2^{k-2}\nonumber
\end{align}
and extend this complete sub-addition chain by adjoining the sequence 
\begin{align}
2^{k-1}+2^{\lfloor\frac{k-1}{2}\rfloor},2^{k-1}+2^{\lfloor\frac{k-1}{2}\rfloor}+2^{\lfloor\frac{k-1}{2^2}\rfloor},\ldots,2^{k-1}+2^{\lfloor\frac{k-1}{2}\rfloor}+2^{\lfloor\frac{k-1}{2^2}\rfloor}+\cdots+2^1.\nonumber
\end{align}
Since $\xi(n,s)=0$ if $n=2^r$ and $0\leq \xi(n,s)<1$ if $n\neq 2^{r}$, we note that the adjoined sequence contributes at most 
\begin{align}
\left\lfloor\frac{\log k}{\log 2}\right\rfloor =\left\lfloor\frac{\log (\frac{n}{2^s}-\xi(n,s))}{\log 2}\right\rfloor=\left\lfloor\frac{\log n-s\log 2}{\log 2}\right\rfloor=\left\lfloor\frac{\log n}{\log 2}\right\rfloor-s\nonumber
\end{align}
terms to the original complete sub-addition chain. Due to the inequality
\begin{align}
2^{k-1}+2^{\lfloor\frac{k-1}{2}\rfloor}+2^{\lfloor\frac{k-1}{2^2}\rfloor}+\cdots+2^1&<\sum \limits_{i=1}^{k-1}2^i\nonumber \\&=2^k-2\nonumber
\end{align}
we insert terms into the sum
\begin{align}
2^{k-1}+2^{\lfloor\frac{k-1}{2}\rfloor}+2^{\lfloor\frac{k-1}{2^2}\rfloor}+\cdots+2^1\label{secondcanalybroken road}
\end{align}
so that
\begin{align}
\sum\limits_{i=1}^{k-1}2^i=2^k-2.\nonumber 
\end{align}
We analyze the cost of filling in the missing terms of the underlying sum. We insert $2^{k-2}+2^{k-3}+\cdots +2^{\lfloor \frac{k-1}{2}\rfloor+1}$ into \eqref{secondcanalybroken road} and the number of terms adjoined is at most 
\begin{align}
k-2-\left\lfloor\frac{k-1}{2}\right\rfloor.\nonumber
\end{align}
The last term of the adjoined sequence is
\begin{align}
2^{k-1}+(2^{k-2}+2^{k-3}+\cdots+2^{\lfloor \frac{k-1}{2}\rfloor+1})+2^{\lfloor\frac{k-1}{2}\rfloor}+2^{\lfloor\frac{k-1}{2^2}\rfloor}+\cdots+2^1.\label{secondcanalybroken road 1}
\end{align}
Again, we insert $2^{\lfloor\frac{k-1}{2}\rfloor-1}+\cdots+2^{\lfloor \frac{k-1}{2^2}\rfloor+1}$ into \eqref{secondcanalybroken road 1} and the number of terms adjoined is at most
\begin{align}
\left\lfloor\frac{k-1}{2}\right\rfloor-\left\lfloor\frac{k-1}{2^2}\right\rfloor-1.\nonumber
\end{align}
The last term of the adjoined sequence is
\begin{align}
2^{k-1}+(2^{k-2}+2^{k-3}+\cdots +2^{\lfloor \frac{k-1}{2}\rfloor+1})+2^{\lfloor \frac{k-1}{2}\rfloor}+(2^{\lfloor \frac{k-1}{2}\rfloor-1}+\cdots+2^{\lfloor \frac{k-1}{2^2}\rfloor+1})+2^{\lfloor \frac{k-1}{2^2}\rfloor}+\nonumber \\ \cdots+2^1.\label{secondcanalybroken road 2}
\end{align}
Iterating the process, we insert in the immediately previous term by inserting into \eqref{secondcanalybroken road 2} and the number of terms adjoined is at most
\begin{align}
\left\lfloor\frac{k-1}{2^{j}}\right\rfloor-\left\lfloor\frac{k-1}{2^{j+1}}\right\rfloor-1\nonumber
\end{align}
for $j\leq \lfloor\frac{\log n}{\log 2}\rfloor-s$, since we are filling at most $\lfloor\frac{\log k}{\log 2}\rfloor$ blocks with $k=\frac{n}{2^s}-\xi(n,s)$. The contribution of these new terms is at most
\begin{align}
k-1-\left \lfloor\frac{k-1}{2^{\lfloor\frac{\log k}{\log 2}\rfloor}}\right\rfloor-\left\lfloor\frac{\log k}{\log 2}\right\rfloor 
\end{align}
obtained by adding the numbers in the chain 
\begin{align}
k-1-\left\lfloor\frac{k-1}{2}\right\rfloor-1 \nonumber
\end{align}
\begin{align}
\left\lfloor\frac{k-1}{2}\right\rfloor-\left\lfloor\frac{k-1}{2^2}\right\rfloor-1\nonumber
\end{align}
\begin{align}
\vdots \vdots \vdots \vdots \vdots \vdots \vdots \vdots \vdots \vdots \vdots \vdots \nonumber
\end{align}
\begin{align}
\vdots \vdots \vdots \vdots \vdots \vdots \vdots \vdots \vdots \vdots \vdots \vdots \nonumber
\end{align}
\begin{align}
\left\lfloor\frac{k-1}{2^{\lfloor\frac{\log k}{\log 2}\rfloor}}\right\rfloor-\left\lfloor\frac{k-1}{2^{\lfloor\frac{\log k}{\log 2}\rfloor+1}}\right\rfloor-1.\nonumber
\end{align}
By a quick book-keeping, we deduce with $k=\frac{n}{2^s}-\xi(n,s)$ for the length
\begin{align}
\delta(2^{k}-1)&\leq k+k-1-\left\lfloor\frac{k-1}{2^{\lfloor \frac{\log k}{\log 2}\rfloor+1}}\right\rfloor-\left\lfloor\frac{\log k}{\log 2}\right\rfloor+\left\lfloor\frac{\log n}{\log 2}\right\rfloor-s\nonumber \\&\leq \frac{n}{2^{s-1}}-1-\left \lfloor\frac{\frac{n}{2^s}-\xi(n,s)-1}{2^{\lfloor \frac{\log n}{\log 2}\rfloor+1-s}}\right\rfloor-\left\lfloor\frac{\log n}{\log 2}\right\rfloor+s +\left\lfloor\frac{\log n}{\log 2}\right\rfloor-s\nonumber \\&=\frac{n}{2^{s-1}}-1-\left\lfloor\frac{\frac{n}{2^s}-\xi(n,s)-1}{2^{\lfloor\frac{\log n}{\log 2}\rfloor+1-s}}\right \rfloor.\label{secondcanalylast analysis}
\end{align}
Plugging the inequality \eqref{secondcanalylast analysis} into the inequalities in \eqref{secondcanalymajor inequality} and noting that $\ell(\cdot)\leq \delta(\cdot)$, we obtain
\begin{align}
\ell(2^n-1)&\leq \sum\limits_{v=1}^{s}\frac{n}{2^v}+s+2\sum \limits_{j=1}^{s}\sum\limits_{\substack{\eta(2^m-1)\neq 0\\m=\lfloor \frac{n}{2^{j-1}}\rfloor}}1-\theta(n,s)+\ell(2^{\frac{n}{2^s}-\xi(n,s)}-1)\nonumber \\&=n(1-\frac{1}{2^{s}})+\frac{n}{2^{s-1}}-1+s+2\sum \limits_{j=1}^{s}\sum \limits_{\substack{\eta(2^m-1)\neq 0\\m=\lfloor \frac{n}{2^{j-1}}\rfloor }}1-\theta(n,s)-\left\lfloor\frac{\frac{n}{2^s}-\xi(n,s)-1}{2^{\lfloor\frac{\log n}{\log 2}\rfloor+1-s}}\right \rfloor \nonumber \\&=n-1+\frac{n}{2^s}+s+2\sum \limits_{j=1}^{s}\sum \limits_{\substack{\eta(2^m-1)\neq 0\\m=\lfloor\frac{n}{2^{j-1}}\rfloor}}1-\theta(n,s)-\left\lfloor \frac{\frac{n}{2^s}-\xi(n,s)-1}{2^{\lfloor \frac{\log n}{\log 2}\rfloor+1-s}}\right\rfloor\nonumber
\end{align}
where 
$$
\sum \limits_{j=1}^{s}\sum \limits_{\substack{\eta(2^m-1)\neq 0\\m=\lfloor\frac{n}{2^{j-1}}\rfloor}}1
$$ 
counts the number of carry up to the $s$ level for the number $2^n-1$. Taking $2\leq s:=s(n)$ such that $s=\lfloor\frac{\log n}{\log 2}\rfloor$ which is the maximum frequency of the iteration, we deduce 
$$
\left\lfloor\frac{\frac{n}{2^s}-\xi(n,s)-1}{2^{\lfloor\frac{\log n}{\log 2}\rfloor+1-s}}\right\rfloor=0
$$ 
and get 
$$
\sum\limits_{j=1}^{s}\sum\limits_{\substack{\eta(2^m-1)\neq 0\\m=\lfloor\frac{n}{2^{j-1}}\rfloor}}1\leq \frac{1}{2(1+c)}\left\lfloor\frac{\log n}{\log 2}\right\rfloor-1
$$ 
and the inequality 
$$
\ell(2^n-1)\leq n-1-\theta \left(n,\left\lfloor\frac{\log n}{\log 2}\right\rfloor\right)+\left\lfloor\frac{\log n}{\log 2}\right\rfloor+2+\frac{2}{(1+\log n)}\left\lfloor\frac{\log n}{\log 2}\right\rfloor-2
$$ 
for $\theta(n,\lfloor \frac{\log n}{\log 2}\rfloor):=\sum \limits_{j=1}^{\lfloor \frac{\log n}{\log 2}\rfloor}\xi(n,j)>0$ with $n\geq 4$. The claimed inequality follows as a consequence. 
\end{proof}
\bigskip

It follows similarly from the proofs the following result which holds for numbers of the form $2^n-1$ with low degree of carry 
$$
\frac{1}{(1+\log\log n)}\left\lfloor\frac{\log n}{\log 2}\right\rfloor-1
$$
\bigskip

\begin{theorem}
If $2^n-1$ has carry of degree at most 
$$
\kappa(2^n-1):=\left(\frac{1}{1+\log\log n}\right)\left\lfloor \frac{\log n}{\log 2}\right\rfloor-1
$$ 
then 
$$
\ell(2^n-1)\leq n-1+\left(1+\frac{2}{1+\log \log n}\right)\left\lfloor\frac{\log n}{\log 2}\right\rfloor
$$ holds for all $n\in \mathbb{N}$ with $n\geq 4$, where $\ell(\cdot)$ denotes the length of the shortest addition chain that leads to $\cdot$.
\end{theorem}
\bigskip

\begin{theorem}
If $2^n-1$ has carry of degree at most 
$$
\kappa(2^n-1):=\left(\frac{1}{1+(\log\log n)^2}\right)\left\lfloor\frac{\log n}{\log 2}\right\rfloor-1
$$ 
then 
$$
\ell(2^n-1)\leq n-1+\left(1+\frac{2}{1+(\log\log n)^2}\right)\left\lfloor\frac{\log n}{\log 2}\right\rfloor
$$ 
for all $n\in \mathbb{N}$ with $n\geq 4$, where $\ell(\cdot)$ denotes the length of the shortest addition chain that leads to $\cdot$.
\end{theorem}
\bigskip

We obtain a more general theorem. 

\begin{theorem}\label{general result}
If $2^n-1$ has carry of degree at most 
$$
\kappa(2^n-1):=\left(\frac{1}{1+f(n)}\right)\left\lfloor\frac{\log n}{\log 2}\right\rfloor-1
$$ 
where $f(n)=o(\log n)$ with $f(n)\longrightarrow \infty$ as $n\longrightarrow \infty$, then
$$
\ell(2^n-1)\leq n-1+\left(1+\frac{2}{1+f(n)}\right)\left\lfloor\frac{\log n}{\log 2}\right\rfloor
$$ 
for all $n\in \mathbb{N}$ with $n\geq 4$, where $\ell(\cdot)$ denotes the length of the shortest addition chain that leads to $\cdot$.
\end{theorem}
\bigskip

The following chain of results obtained illustrates that to make progress on the Scholz conjecture, it suffices to study possible ways of controlling numbers of the form $2^n-1$ with high carries. In other words, the degree of carry of numbers of the form $2^n-1$ determines the quality of the upper bound for the corresponding length of the shortest addition using the current method. We end the paper using the same method with the expected number of carries for a fixed number $2^n-1$, we can obtain the stronger result:

\begin{theorem}\label{expected carries result}
If $2^n-1$ has carry of degree at most 
$$
\kappa(2^n-1)=\frac{1}{2}\left(\ell(n)-\left\lfloor\frac{\log n}{\log 2}\right\rfloor+\sum \limits_{j=1}^{\lfloor\frac{\log n}{\log 2}\rfloor}\left\{\frac{n}{2^j}\right\}\right)
$$ 
then
$$
\ell(2^n-1)\leq n+1+\sum\limits_{j=1}^{\lfloor\frac{\log n}{\log 2}\rfloor}\bigg(\left\{\frac{n}{2^j}\right\}-\xi(n,j)\bigg)+\ell(n)
$$
for all $n\in \mathbb{N}$ with $n\geq 4$, where $\ell(\cdot)$ denotes the length of the shortest addition chain producing $\cdot$, $\{\cdot\}$ denotes the fractional part of $\cdot$ and where $\xi(n,1):=\{\frac{n}{2}\}$ with $\xi(n,2)=\{\frac{1}{2}\lfloor \frac{n}{2}\rfloor\}$ and so on.
\end{theorem}

\begin{proof}
  Suppose that $2^n-1$ has at most 
  $$
  \frac{1}{2}\left(\ell(n)-\left\lfloor\frac{\log n}{\log 2}\right\rfloor+\sum \limits_{j=1}^{\lfloor\frac{\log n}{\log 2}\rfloor}\left\{\frac{n}{2^j}\right\}\right)
  $$ 
  degrees of carry. We write
  $$
  2^n-1=(2^{\lfloor\frac{n}{2}\rfloor}-1)(2^{\lfloor\frac{n}{2}\rfloor }+1)+\eta(2^n-1)
  $$
  where 
  $$
  \eta(2^n-1):=\frac{(1-(-1)^n)}{2}(2^{n-(1-(-1)^n)\frac{1}{2}})
  $$ 
  is the level one carry of $2^n-1.$ We can recover the general factorization of $2^n-1$ from this identity according to the parity of the exponent $n$. In particular, if $n\equiv 0\pmod 2$, then
  $$
  2^n-1=(2^{\frac{n}{2}}-1)(2^{\frac{n}{2}}+1)
  $$
  and 
  $$
  2^n-1=(2^{\frac{n-1}{2}}-1)(2^{\frac{n-1}{2}}+1)+2^{n-1}
  $$ 
  if $n\equiv 1\pmod 2$. Combining both cases, we obtain 
  $$
  \ell(2^n-1)\leq \ell\left((2^{\lfloor\frac{n}{2}\rfloor }-1)(2^{\lfloor\frac{n}{2}\rfloor}+1)\right)+\eta(2^n-1).
  $$ 
  Applying the lemma \ref{Alfred Braurer}, we obtain further the inequality
\begin{align}
\ell(2^n-1)\leq \ell(2^{\lfloor\frac{n}{2}\rfloor }-1)+\ell(2^{\lfloor\frac{n}{2}\rfloor}+1)+\eta(2^n-1)\label{tcanalyfirst decomp}
\end{align}
We set $\lfloor\frac{n}{2}\rfloor=k$ in \eqref{tcanalyfirst decomp} and write
$$
2^k-1=(2^{\lfloor\frac{k}{2}\rfloor}-1)(2^{\lfloor\frac{k}{2}\rfloor}+1)+\eta(2^k-1)
$$ 
where 
$$
\eta(2^k-1)=\frac{(1-(-1)^k)}{2}(2^{k-(1-(-1)^k)\frac{1}{2}})
$$ 
is the carry of $2^k-1$. We can recover the general factorization of $2^k-1$ from this identity according to the parity of the exponent $k$. In particular, if $k\equiv 0\pmod 2$, then 
$$
2^k-1=(2^{\frac{k}{2}}-1)(2^{\frac{k}{2}}+1)
$$ 
and 
$$
2^k-1=(2^{\frac{k-1}{2}}-1)(2^{\frac{k-1}{2}}+1)+2^{k-1}
$$ 
if $k\equiv 1\pmod 2$. Combining both cases, we obtain the inequality 
$$
\ell(2^k-1)\leq \ell((2^{\lfloor\frac{k}{2}\rfloor}-1)(2^{\lfloor \frac{k}{2}\rfloor}+1))+\eta(2^k-1).
$$ 
Applying the lemma \ref{Alfred Braurer}, we get
\begin{align}
\ell(2^k-1)&\leq \ell(2^{\lfloor\frac{k}{2}\rfloor }-1)+\ell(2^{\lfloor\frac{k}{2}\rfloor}+1)+\eta(2^k-1)\nonumber \\&=\ell(2^{\lfloor\frac{1}{2}\lfloor\frac{n}{2}\rfloor \rfloor}-1)+\ell(2^{\lfloor\frac{1}{2}\lfloor\frac{n}{2}\rfloor \rfloor}+1)+\eta(2^{\lfloor\frac{n}{2}\rfloor}-1)\label{tcanalysecond decomp}
\end{align}
so that by inserting \eqref{tcanalysecond decomp} into \eqref{tcanalyfirst decomp}, we obtain
\begin{align}
\ell(2^n-1)&\leq \ell(2^{\lfloor\frac{1}{2}\lfloor\frac{n}{2}\rfloor\rfloor}-1)+\ell(2^{\lfloor\frac{1}{2}\lfloor\frac{n}{2}\rfloor\rfloor}+1)+\eta(2^{\lfloor\frac{n}{2}\rfloor}-1)\nonumber \\&+\ell(2^{\lfloor\frac{n}{2}\rfloor }+1)+\eta(2^n-1).\label{tcanalycombined}
\end{align}
We iterate the factorization to frequency $s$ to obtain 
\begin{align}
\ell(2^n-1)&\leq \ell(2^{\lfloor\frac{n}{2}\rfloor }+1)+\eta(2^n-1)+\ell(2^{\lfloor\frac{1}{2}\lfloor\frac{n}{2}\rfloor \rfloor}-1)+\ell(2^{\lfloor\frac{1}{2}\lfloor\frac{n}{2}\rfloor \rfloor}+1)+\eta(2^{\lfloor\frac{n}{2}\rfloor}-1)\nonumber \\&+\cdots+\ell(2^{\frac{n}{2^s}-\xi(n,s)}-1)+\ell(2^{\frac{n}{2^s}-\xi(n,s)}+1)+\eta(2^{\lfloor \frac{n}{2^{s-1}}\rfloor}-1)\label{tcanalylength bound}
\end{align}
where $0\leq \xi(n,s)<1$ for a fixed integer $2\leq s:=s(n)$ to be chosen later. For example
$$
\xi(n,1)=(1-(-1)^n)\frac{1}{4}<1
$$ 
and 
$$
\xi(n,2)=(1-(-1)^n)\frac{1}{8}+(1-(-1)^k)\frac{1}{4}<1
$$
with 
$$
k:=\left\lfloor\frac{n}{2}\right\rfloor
$$ 
and so on. The function $\xi(n,s)$ for values of $s\geq 3$ can be read from exponents of the terms arising from the iteration process. We deduce from \eqref{tcanalylength bound} the inequality 
\begin{align}
\ell(2^n-1)&\leq \sum\limits_{v=1}^{s}\frac{n}{2^v}+s+2\sum \limits_{j=1}^{s}\sum\limits_{\substack{\eta(2^m-1)\neq 0\\m=\lfloor \frac{n}{2^{j-1}}\rfloor }}1-\theta(n,s)+\ell(2^{\frac{n}{2^s}-\xi(n,s)}-1)\nonumber \\&=n(1-\frac{1}{2^{s}})+s+2\sum \limits_{j=1}^{s}\sum \limits_{\substack{\eta(2^m-1)\neq 0\\m=\lfloor \frac{n}{2^{j-1}}\rfloor }}1-\theta(n,s)+\ell(2^{\frac{n}{2^s}-\xi(n,s)}-1)\label{tcanalymajor inequality}
\end{align}
where the term 
$$
\sum \limits_{j=1}^{s}\sum \limits_{\substack{\kappa(2^m-1)\neq 0\\m=\lfloor\frac{n}{2^{j-1}}\rfloor}}1
$$ 
counts the number of all carry of $2^n-1$ up to level $s$ and $0\leq \theta(n,s):=\sum \limits_{j=1}^{s}\xi(n,j)$ and $2\leq s:=s(n)$, a fixed integer to be chosen later. We note that 
$$
\theta(n,s):=\sum \limits_{j=1}^{s}\xi(n,j)=0
$$ 
if $n=2^r$ for some $r\in \mathbb{N}$, since $\xi(n,j)=0$ for each $1\leq j\leq s$ for all $n$ that are powers of $2$. It is also important to note that the $2s$ term is obtained by noting that there are at most $s$ terms with odd exponents under the iteration process and each term with odd exponent contributes $2$, and the other $s$ term comes from summing $1$ with frequency $s$ finding the total length of the short addition chains producing numbers of the form $2^v+1$. Now, we set $k=\frac{n}{2^s}-\xi(n,s)$ and construct the addition chain producing $2^{k}$
$$
1,2,2^2,\ldots, 2^{k-1},2^k
$$ 
with the corresponding partition sequence
\begin{align}
2=1+1,2+2=2^2,2^2+2^2=2^3,\ldots,2^{k-1}=2^{k-2}+2^{k-2},2^k=2^{k-1}+2^{k-1}\nonumber
\end{align}
with $a_i=2^{i-2}=r_i$ for $2\leq i\leq k+1$, where $a_i$ and $r_i$ denote the determiner and regulator of the $i^{th}$ generator of the chain. We consider only the complete sub-addition chain 
\begin{align}
2=1+1,2+2=2^2,2^2+2^2=2^3,\ldots,2^{k-1}=2^{k-2}+2^{k-2}\nonumber
\end{align}
and extend this complete sub-addition chain by adjoining the sequence 
\begin{align}
2^{k-1}+2^{\lfloor\frac{k-1}{2}\rfloor},2^{k-1}+2^{\lfloor\frac{k-1}{2}\rfloor}+2^{\lfloor\frac{k-1}{2^2}\rfloor},\ldots,2^{k-1}+2^{\lfloor\frac{k-1}{2}\rfloor}+2^{\lfloor\frac{k-1}{2^2}\rfloor}+\cdots+2^1.\nonumber
\end{align}
Since $\xi(n,s)=0$ if $n=2^r$ and $0\leq \xi(n,s)<1$ if $n\neq 2^{r}$, we note that the adjoined sequence contributes at most 
\begin{align}
\left\lfloor\frac{\log k}{\log 2}\right\rfloor =\left\lfloor\frac{\log (\frac{n}{2^s}-\xi(n,s))}{\log 2}\right\rfloor=\left\lfloor\frac{\log n-s\log 2}{\log 2}\right\rfloor=\left\lfloor\frac{\log n}{\log 2}\right\rfloor-s\nonumber
\end{align}
terms to the original complete sub-addition chain. Due to the inequality
\begin{align}
2^{k-1}+2^{\lfloor\frac{k-1}{2}\rfloor}+2^{\lfloor\frac{k-1}{2^2}\rfloor}+\cdots+2^1&<\sum \limits_{i=1}^{k-1}2^i\nonumber \\&=2^k-2\nonumber
\end{align}
we insert terms into the sum
\begin{align}
2^{k-1}+2^{\lfloor \frac{k-1}{2}\rfloor}+2^{\lfloor \frac{k-1}{2^2}\rfloor}+\cdots+2^1\label{tcanalybroken road}
\end{align}
so that
\begin{align}
\sum \limits_{i=1}^{k-1}2^i=2^k-2.\nonumber 
\end{align}
We analyze the cost of filling the missing terms of the underlying sum. We insert $2^{k-2}+2^{k-3}+\cdots +2^{\lfloor \frac{k-1}{2}\rfloor+1}$ into \eqref{tcanalybroken road} and the number of terms adjoined is at most  
\begin{align}
k-2-\left\lfloor\frac{k-1}{2}\right\rfloor \nonumber
\end{align}
The last term of the adjoined sequence is
\begin{align}
2^{k-1}+(2^{k-2}+2^{k-3}+\cdots+2^{\lfloor\frac{k-1}{2}\rfloor+1})+2^{\lfloor\frac{k-1}{2}\rfloor}+2^{\lfloor\frac{k-1}{2^2}\rfloor}+\cdots+2^1.\label{tcanalybroken road 1}
\end{align}
Again, we insert $2^{\lfloor\frac{k-1}{2}\rfloor-1}+\cdots+2^{\lfloor\frac{k-1}{2^2}\rfloor+1}$ into \eqref{tcanalybroken road 1} and the number of terms adjoined is at most
\begin{align}
\left\lfloor\frac{k-1}{2}\right\rfloor-\left\lfloor\frac{k-1}{2^2}\right\rfloor-1.\nonumber
\end{align}
The last term of the adjoined sequence is  
\begin{align}
2^{k-1}+(2^{k-2}+2^{k-3}+\cdots+2^{\lfloor\frac{k-1}{2}\rfloor+1})+2^{\lfloor\frac{k-1}{2}\rfloor}+(2^{\lfloor\frac{k-1}{2}\rfloor-1}+\cdots+2^{\lfloor\frac{k-1}{2^2}\rfloor+1})+2^{\lfloor\frac{k-1}{2^2}\rfloor}+\nonumber \\ \cdots+2^1.\label{tcanalybroken road 2}
\end{align}
Iterating the process, we insert in the immediately previous term by inserting into \eqref{tcanalybroken road 2} and the number of terms adjoined is at most
\begin{align}
\left\lfloor\frac{k-1}{2^{j}}\right\rfloor-\left\lfloor\frac{k-1}{2^{j+1}}\right\rfloor-1\nonumber
\end{align}
for $j\leq \lfloor\frac{\log n}{\log 2}\rfloor-s$, since we are filling at most $\lfloor\frac{\log k}{\log 2}\rfloor$ blocks with $k=\frac{n}{2^s}-\xi(n,s)$. The contribution of these new terms is at most
\begin{align}
k-1-\left\lfloor\frac{k-1}{2^{\lfloor\frac{\log k}{\log 2}\rfloor}}\right\rfloor-\left\lfloor\frac{\log k}{\log 2}\right\rfloor 
\end{align}
obtained by adding the numbers in the chain 
\begin{align}
k-1-\left\lfloor\frac{k-1}{2}\right\rfloor-1 \nonumber
\end{align}
\begin{align}
\left\lfloor\frac{k-1}{2}\right\rfloor-\left\lfloor\frac{k-1}{2^2}\right\rfloor-1\nonumber
\end{align}
\begin{align}
\vdots \vdots \vdots \vdots \vdots \vdots \vdots \vdots \vdots \vdots \vdots \vdots \nonumber
\end{align}
\begin{align}
\vdots \vdots \vdots \vdots \vdots \vdots \vdots \vdots \vdots \vdots \vdots \vdots \nonumber
\end{align}
\begin{align}
\left\lfloor\frac{k-1}{2^{\lfloor\frac{\log k}{\log 2}\rfloor}}\right\rfloor-\left\lfloor\frac{k-1}{2^{\lfloor\frac{\log k}{\log 2}\rfloor+1}}\right\rfloor-1.\nonumber
\end{align}
By a quick book-keeping, we deduce with $k=\frac{n}{2^s}-\xi(n,s)$ for the length
\begin{align}
\delta(2^{k}-1)&\leq k+k-1-\left\lfloor\frac{k-1}{2^{\lfloor \frac{\log k}{\log 2}\rfloor+1}}\right\rfloor-\left\lfloor\frac{\log k}{\log 2}\right\rfloor+\left\lfloor\frac{\log n}{\log 2}\right\rfloor-s\nonumber \\&\leq \frac{n}{2^{s-1}}-1-\left\lfloor\frac{\frac{n}{2^s}-\xi(n,s)-1}{2^{\lfloor\frac{\log n}{\log 2}\rfloor+1-s}}\right\rfloor-\left\lfloor\frac{\log n}{\log 2}\right\rfloor+s +\left\lfloor\frac{\log n}{\log 2}\right\rfloor-s\nonumber \\&=\frac{n}{2^{s-1}}-1-\left \lfloor \frac{\frac{n}{2^s}-\xi(n,s)-1}{2^{\lfloor \frac{\log n}{\log 2}\rfloor+1-s}}\right \rfloor.\label{tcanalylast analysis}
\end{align}
Plugging the inequality \eqref{tcanalylast analysis} into the inequalities in \eqref{tcanalymajor inequality} and noting that $\ell(\cdot)\leq \delta(\cdot)$, we obtain
\begin{align}
\ell(2^n-1)&\leq \sum\limits_{v=1}^{s}\frac{n}{2^v}+s+2\sum \limits_{j=1}^{s}\sum\limits_{\substack{\eta(2^m-1)\neq 0\\m=\lfloor\frac{n}{2^{j-1}}\rfloor}}1-\theta(n,s)+\ell(2^{\frac{n}{2^s}-\xi(n,s)}-1)\nonumber \\&=n(1-\frac{1}{2^{s}})+\frac{n}{2^{s-1}}-1+s+2\sum\limits_{j=1}^{s}\sum \limits_{\substack{\eta(2^m-1)\neq 0\\m=\lfloor \frac{n}{2^{j-1}}\rfloor }}1-\theta(n,s)-\left\lfloor\frac{\frac{n}{2^s}-\xi(n,s)-1}{2^{\lfloor\frac{\log n}{\log 2}\rfloor+1-s}}\right \rfloor \label{fine} \\&=n-1+\frac{n}{2^s}+s+2\sum \limits_{j=1}^{s}\sum\limits_{\substack{\eta(2^m-1)\neq 0\\m=\lfloor\frac{n}{2^{j-1}}\rfloor}}1-\theta(n,s)-\left \lfloor\frac{\frac{n}{2^s}-\xi(n,s)-1}{2^{\lfloor\frac{\log n}{\log 2}\rfloor+1-s}}\right\rfloor\nonumber
\end{align}
where
$$
\sum\limits_{j=1}^{s}\sum\limits_{\substack{\eta(2^m-1)\neq 0\\m=\lfloor \frac{n}{2^{j-1}}\rfloor}}1
$$ 
counts the number of carry up to the $s$ level for the number $2^n-1$. Taking $2\leq s:=s(n)$ such that $s= \lfloor\frac{\log n}{\log 2}\rfloor$ which is the maximum frequency of the iteration, we get
$$
\left\lfloor\frac{\frac{n}{2^s}-\xi(n,s)-1}{2^{\lfloor\frac{\log n}{\log 2}\rfloor+1-s}}\right\rfloor=0
$$ 
and
$$
\sum\limits_{j=1}^{s}\sum\limits_{\substack{\eta(2^m-1)\neq 0\\m=\lfloor\frac{n}{2^{j-1}}\rfloor}}1\leq \frac{1}{2}\left(\ell(n)-\left\lfloor\frac{\log n}{\log 2}\right\rfloor+\sum \limits_{j=1}^{\lfloor\frac{\log n}{\log 2}\rfloor}\left\{\frac{n}{2^j}\right\}\right)
$$ 
and deduce
$$
\ell(2^n-1)\leq n-1-\theta \left(n,\left\lfloor\frac{\log n}{\log 2}\right\rfloor\right)+\left\lfloor\frac{\log n}{\log 2}\right\rfloor+2+\ell(n)-\left\lfloor\frac{\log n}{\log 2}\right\rfloor+\sum\limits_{j=1}^{\lfloor\frac{\log n}{\log 2}\rfloor}\left\{\frac{n}{2^j}\right\}
$$ 
for $\theta(n,\lfloor\frac{\log n}{\log 2}\rfloor):=\sum \limits_{j=1}^{\lfloor\frac{\log n}{\log 2}\rfloor}\xi(n,j)>0$ with $n\geq 4$ and $0\leq \xi(n,j)<1$, where $\{\cdot \}$ denotes the fractional part of a real number. The claimed inequality follows as a consequence. 
\end{proof}
\bigskip

It turns out that proving that integers of the form $2^n-1$ has carry of degree at most 
$$
\kappa(2^n-1)=\frac{1}{2}\left(\ell(n)-\left\lfloor\frac{\log n}{\log 2}\right\rfloor+\sum \limits_{j=1}^{\lfloor \frac{\log n}{\log 2}\rfloor}\left\{\frac{n}{2^j}\right\}\right)
$$ 
yields the inequality 
$$
\ell(2^n-1)\leq n+1+\sum\limits_{j=1}^{\lfloor\frac{\log n}{\log 2}\rfloor}\bigg(\left\{\frac{n}{2^j}\right\}-\xi(n,j)\bigg)+\ell(n)
$$ 
for all $n\in \mathbb{N}$ as shown in the preceding proof, which is slightly short of the original conjecture. We expect the degree of carry of all integers to be of the form above since for integers of the form $n=2^k$, it matches exactly with degree. In particular, the number $2^{2^k}-1$ is always free from carry and  $$
\kappa(2^{2^k}-1)=\frac{1}{2}\left(\ell(2^k)-\left\lfloor\frac{\log 2^k}{\log 2}\right\rfloor+\sum \limits_{j=1}^{\lfloor\frac{\log n}{\log 2}\rfloor}\left\{\frac{2^k}{2^j}\right\}\right)=0.
$$ 
We make the following conjecture 

\begin{conjecture}(The carry-addition chain conjecture)
Let $n\geq 2$ and let $\kappa(\cdot)$ denote the degree of carry of $\cdot$. We have 
$$
\kappa(2^n-1)=\left\lceil\frac{1}{2}\left(\ell(n)-\left\lfloor \frac{\log n}{\log 2}\right\rfloor+\sum \limits_{j=1}^{\lfloor \frac{\log n}{\log 2}\rfloor}\left\{\frac{n}{2^j}\right\}\right)\right\rceil
$$ 
where $\lceil \cdot \rceil$ denotes the ceiling of $\cdot$.
\end{conjecture}


\begin{figure}[ht]
\centering
\makebox[\textwidth][c]{%
\begin{tikzpicture}[
    >=Latex,
    font=\large,
    body/.style={
        draw,
        very thick,
        rounded corners=8pt,
        minimum width=15.8cm,
        minimum height=8.4cm,
        fill=gray!4
    },
    port/.style={
        draw,
        very thick,
        circle,
        minimum size=1.55cm,
        inner sep=0pt,
        fill=white
    },
    panel/.style={
        draw,
        thick,
        rounded corners=5pt,
        minimum width=4.25cm,
        minimum height=2.15cm,
        align=center,
        fill=white
    },
    display/.style={
        draw,
        thick,
        rounded corners=5pt,
        minimum width=4.6cm,
        minimum height=2.15cm,
        align=center,
        fill=white
    },
    pipe/.style={->, very thick}
]

\node[body] (machine) at (0,0) {};

\node[font=\LARGE\bfseries] at (0,3.55) {Carry Analysis Machine};

\node[port] (input) at (-8.0,0) {$c$};
\node[font=\normalsize, align=center] at (-8.0,-1.35)
{degree of carry};

\node[port] (outputport) at (8.0,0) {};
\node[font=\normalsize, align=center] at (8.0,-1.35)
{bound output};

\node[panel] (dial) at (-3.95,0.15) {%
\textbf{Carry degree detector}\\[2pt]
\begin{tikzpicture}[baseline=-0.6ex,scale=0.78]
    \draw[thick] (0,0) arc (180:0:1.65);
    \draw[thick] (-1.65,0) -- (-1.35,0);
    \draw[thick] (1.65,0) -- (1.35,0);
    \draw[thick] (-1.17,0.85) -- (-0.98,0.58);
    \draw[thick] (0,1.08) -- (0,0.76);
    \draw[thick] (1.17,0.85) -- (0.98,0.58);
    \draw[very thick] (0,0) -- (0.95,0.72);
    \fill (0,0) circle (2pt);
    \node[font=\scriptsize] at (-1.65,-0.35) {low};
    \node[font=\scriptsize] at (1.65,-0.35) {high};
\end{tikzpicture}

{\small low degree $\Rightarrow$ sharper bound;\quad high degree $\Rightarrow$ weaker bound}
};

\node[panel] (chamber) at (0.05,0.15) {%
\textbf{Propagation chamber}\\[2pt]
\parbox{3.55cm}{\centering
The incoming carry is redistributed locally and converted into a controlled estimate.}
};

\node[display] (bound) at (4.95,0.15) {%
\textbf{Bound register}\\[2pt]
$\displaystyle \ell(2^n-1)\le \mathcal{B}(c)$\\[4pt]
{\small the quality of $\mathcal{B}(c)$ depends on the carry degree}
};

\draw[pipe] (input.east) -- ++(0.85,0) |- (dial.west);
\draw[pipe] (dial.east) -- (chamber.west);
\draw[pipe] (chamber.east) -- (bound.west);
\draw[pipe] (bound.east) -- ++(0.85,0) |- (outputport.west);

\node[draw, thick, rounded corners=5pt, fill=white, minimum width=13.7cm, minimum height=1.05cm, align=center]
at (0,-2.95)
{\large
A lower carry degree produces a stronger estimate for the minimal addition-chain length of $2^n-1$; a higher carry degree produces a weaker estimate.
};

\end{tikzpicture}%
}
\caption{A machine-style schematic of carry analysis as a bound-producing mechanism.}
\end{figure}

\footnote{
\par
.}%
.

\bibliographystyle{amsplain}

\end{document}